\documentclass[11pt,reqno]{amsart}

\usepackage[a4paper,top=3cm,bottom=3cm,left=2.5cm,right=2.5cm,marginparwidth=1.75cm]{geometry}
\usepackage[UKenglish]{babel}

\usepackage{amsfonts,amsmath,amsthm,amssymb} % \text and \mathbb command
\usepackage[foot]{amsaddr}
\usepackage{mathrsfs} % script type font
\usepackage{stmaryrd} % new bracket

\RequirePackage[OT1]{fontenc}
\usepackage{natbib}

\usepackage{comment} % commenting out multiple lines
\usepackage{graphicx} % graphics
\usepackage{setspace}
\usepackage{color}

\allowdisplaybreaks[1]

\newcommand{\R}{\mathbf{R}}

\newtheorem{theorem}{Theorem}%[subsection]

\newtheorem{lemma}{Lemma}

\newtheorem{remark}{Remark}

\includecomment{jp-text}
\excludecomment{en-text}
\def\be{\begin{equation}}
\def\ee{\end{equation}}
\def\bea{\begin{eqnarray}}
\def\eea{\end{eqnarray}}
\def\beas{\begin{eqnarray*}}
\def\eeas{\end{eqnarray*}}

\title[Hybrid estimation for noisy ergodic diffusion processes]{Hybrid estimation for ergodic diffusion processes based on noisy discrete observations}

\author[Y Kaino]{Yusuke Kaino$^{1}$}
\author[S H Nakakita]{Shogo H Nakakita$^{1}$}
\author[M Uchida]{Masayuki Uchida$^{1,2,3}$}
\address{$^{1}$Graduate School of Engineering Science, Osaka University}
\address{$^{2}$Center for Mathematical Modeling and Data Science, Osaka University}
\address{$^{3}$JST CREST}


\begin{document}
\maketitle

\begin{abstract}
We consider parametric estimation for ergodic diffusion processes with
noisy sampled data based on the hybrid method, 
%We consider parametric estimation for discretely and noisily observed diffusion processes with hybrid scheme, 
that is, the multi-step estimation with the initial Bayes type estimators.
In order to select proper initial values for optimisation of the quasi likelihood function of
ergodic diffusion processes with noisy observations,
we construct the initial Bayes type estimator based on the local means of the noisy observations.
%which is strongly influential in nonlinear modeling based on stochastic differential equations.
The asymptotic properties of the initial Bayes type estimators and the hybrid multi-step estimators
with the initial Bayes type estimators are shown, and a concrete example and the simulation results are given. 
\end{abstract}

\section{Introduction}

We consider the $d$-dimensional ergodic diffusion process defined by the stochastic differential equation
\begin{align*}
    \mathrm{d}X_{t}=b\left(X_{t},\beta\right)\mathrm{d}t+a\left(X_{t},\alpha\right)\mathrm{d}w_{t},\ X_{0}=x_{0},
\end{align*}
where $\left\{w_{t}\right\}_{t\ge0}$ is an $r$-dimensional Wiener process, $x_{0}$ is a $d$-dimensional random vector independent of $\left\{w_{t}\right\}_{t\ge0}$, $\alpha\in\Theta_{1}$ and $\beta\in\Theta_{2}$ are unknown parameters, $\Theta_{i}\subset\mathbf{R}^{m_{i}}$ is bounded, open and convex sets in $\mathbf{R}^{m_{i}}$ admitting Sobolev's inequalities for embedding $W^{1,p}\left(\Theta_{i}\right)\hookrightarrow C\left(\bar{\Theta_{i}}\right)$ {\citep[see][]{Adams-Fournier-2003, Yoshida-2011}}
for $i=1,2$, 
$\theta^{\star}=\left(\alpha^{\star},\beta^{\star}\right)$ is the true parameter vector, 
and $a:\mathbf{R}^{d}\times\Theta_{1}\to\mathbf{R}^{d}\otimes\mathbf{R}^{r}$ and $b:\mathbf{R}^{d}\times\Theta_{2}\to \mathbf{R}^{d}$ are known functions.

Our concern in this paper is the estimation of $\theta:=\left(\alpha,\beta\right)\in\Theta:=\Theta_{1}\times\Theta_{2}$ with long-term, discrete and noisy observation defined as the sequence of the $d$-dimensional random vectors $\left\{Y_{ih_{n}}\right\}_{i=0,\ldots,n}$ such that for all $i=0,\ldots,n$,
\begin{align*}
    Y_{ih_{n}} := X_{ih_{n}} + \Lambda^{1/2}\varepsilon_{ih_{n}},
\end{align*}
where $h_{n}>0$ is the discretisation step satisfying $h_{n}\to0$ and $T_{n}:=nh_{n}\to\infty$ as $n\to\infty$,  $\left\{\varepsilon_{ih_{n}}\right\}_{i=0,\ldots,n}$ is the i.i.d.\ sequence of $d$-dimensional random vectors with $\mathbf{E}\left[\varepsilon_{0}\right]=\mathbf{0}$ and $\mathrm{Var}\left(\varepsilon_{0}\right)=I_{d}$, 
the components of $\varepsilon_{0}$ are independent of each other and have symmetric distribution with respect to 0, and $\Lambda$ is a $d\times d$ real matrix being positive semi-definite, defining the variance of noise term $\Lambda^{1/2} \varepsilon_{ih_{n}}$. Let us assume the half-vectorisation $\theta_{\varepsilon}:=\mathrm{vech}\Lambda$ is in the bounded, convex and open parameter space $\Theta_{\varepsilon}\subset\mathbf{R}^{d\left(d+1\right)/2}$, and denote $\Xi:=\Theta_{\varepsilon}\times\Theta_{1}\times\Theta_{2}$.

Statistical inference for ergodic diffusion processes
%stochastic differential equations 
%based on sampled data
has been researched
for the last few decades, 
%in these decades 
for instance, see
\citet{Florens-Zmirou-1989,Yoshida-1992,Bibby-Sorensen-1995,Kessler-1995,Kessler-1997,Kutoyants-2004, Iacus-2008, DeGregorioa-Iacus-2013, DeGregorioa-Iacus-2018, Iacus-Yoshida-2018}
and references therein. 
The parametric inference for ergodic diffusion processes with discrete and noisy observations has been researched in \citet{Favetto-2014,Favetto-2016} and 
\citet{Nakakita-Uchida-2017,Nakakita-Uchida-2018a,Nakakita-Uchida-2018b,Nakakita-Uchida-2018c}. For parametric estimation for non-ergodic diffusion processes in the presence of market microstructure noise,
see \citet{Ogihara-2018}.
 \citet{Favetto-2014} proposes a simultaneous quasi likelihood function $\mathbb{H}_{n}\left(\alpha,\beta\right)$ which necessitates optimisation with respect to both $\alpha$ and $\beta$ and shows maximum likelihood (ML) type estimators $\left(\hat{\alpha}_{n},\hat{\beta}_{n}\right)=\arg\max \mathbb{H}_{n}\left(\alpha,\beta\right)$ have consistency even if the variance of noise is unknown; \citet{Favetto-2016} discusses asymptotic normality of the estimator proposed in \citet{Favetto-2014} when the variance of noise is known; \citet{Nakakita-Uchida-2017,Nakakita-Uchida-2018b} suggest adaptive quasi likelihood functions $\mathbb{H}_{1,n}\left(\alpha\right)$ and $\mathbb{H}_{2,n}\left(\beta\right)$ which succeed in lessening the computational burden in comparison to \citet{Favetto-2014,Favetto-2016}, and prove consistency and asymptotic normality of the adaptive ML type estimators corresponding to the quasi likelihoods; \citet{Nakakita-Uchida-2018a} use those quasi likelihood functions for likelihood-ratio-type test and show the asymptotic behaviour of test statistics under both null hypotheses and alternative ones; \citet{Nakakita-Uchida-2018c} analyse those quasi likelihood functions with the framework of quasi likelihood analysis (QLA) proposed by \citet{Yoshida-2011}, and show the polynomial large deviation inequality (PLDI) for the quasi likelihood functions and consequently the convergence of moments of adaptive ML type estimators and adaptive Bayes type ones.
For details of adaptive estimation for diffusion processes, see  
\citet{Yoshida-1992, Yoshida-2011, Uchida-Yoshida-2012, Uchida-Yoshida-2014}.

In general, however, the optimisation of quasi likelihood functions for diffusion processes, regardless of noise existence, is strongly dependent on initial values, especially in the case where the volatility function $a$ or drift function $b$ are nonlinear with respect to parameters. Hence, \citet{Kaino-et-al-2017} and \citet{Kaino-Uchida-2018a, Kaino-Uchida-2018b} propose hybrid multi-step estimation procedure for diffusion processes where initial values in optimisation are derived from Bayes type estimation with reduced sample sizes and the sequential optimisation with these initial values is implemented, which inherits the idea of hybrid multi-step estimation for diffusion processes with full sample sizes in \citet{Kamatani-Uchida-2015} {\citep[see also][]{Kutoyants-2017}}. In this research, we also consider hybrid multi-step estimation 
and apply the idea into inference problem, in particular, PLDI for the quasi likelihood functions and 
the convergence of moments of estimators for discretely and noisily observed ergodic diffusion processes
since PLDI and convergence of moment of estimators are key tools to show the mathematical validity
of information criteria for model selection problems {\citep[see][]{Uchida-2010, Fujii-Uchida-2014, Eguchi-Masuda-2018}}.

This paper consists of the following parts: Section 2 deals with the notation; we define the initial and multi-step estimators and set the main theorem for the polynomial-type large deviation inequalities,
moment estimates of the Bayes type estimators and convergences of moments in Section 3; 
a concrete example and simulation results are given in Section 4,
the conclusions of this work are {summarised} in Section 5,  
and finally we give the proofs of the results in Section 6.

\section{Notation and Assumption}
First of all we give the notation used throughout this paper.

\begin{itemize}
	\item For every matrix $A$, $A^{T}$ is the transpose of $A$, and $A^{\otimes 2}:=AA^{T}$.
	\item For every set of matrices $A$ and $B$ 
%	whose dimensions coincide, 
of the same size,
$A\left[B\right]:=\mathrm{tr}\left(AB^{T}\right)$. Moreover, for any $m\in\mathbf{N}$, $A\in\R^{m}\otimes\R^{m}$ and $u,v\in\R^{m}$, $A\left[u,v\right]:=v^{T}Au$.
	\item Let us denote the $\ell$-th element of any vector $v$ as $v^{\left(\ell\right)}$ and $\left(\ell_{1},\ell_{2}\right)$-th one of any matrix $A$ as $A^{\left(\ell_{1},\ell_{2}\right)}$.
	\item For any vector $v$ and any matrix $A$, $\left|v\right|:=\sqrt{\mathrm{tr}\left(v^{T}v\right)}$ and $\left\|A\right\|:=\sqrt{\mathrm{tr}\left(A^{T}A\right)}$.
	\item For every $p>0$, $\left\|\cdot\right\|_{p}$ is the $L^{p}\left(P_{\theta^{\star}}\right)$-norm.
	\item $A\left(x,\alpha\right):=a\left(x,\alpha\right)^{\otimes 2}$, $a\left(x\right):=a\left(x,\alpha^{\star}\right)$, $A\left(x\right):=A\left(x,\alpha^{\star}\right)$ and $b\left(x\right):=b\left(x,\beta^{\star}\right)$.
	\item For $i=1,2,3$ and $\tau_{i}\in\left(1,2\right]$, $p_{\tau_{i},n}:=h_{n}^{-1/\tau_{i}}$, $\Delta_{\tau_{i},n}:=p_{\tau_{i},n}h_{n}$,  $k_{\tau_{i},n}:=n/p_{\tau_{i},n}=nh_{n}^{1/\tau_{i}}$.
	\item With respect to filtration, for all $i=1,2,3$, $\mathcal{G}_{t}:=\sigma\left(x_{0},w_{s}:s\le t\right)$, $\mathcal{G}_{j,i,n}^{\tau_{i}}:=\mathcal{G}_{j\Delta_{\tau_{i},n}+ih_{n}}$,
	$\mathcal{G}_{j,n}^{\tau_{i}}:=\mathcal{G}_{j,0,n}^{\tau_{i}}$, 
	$\mathcal{A}_{j,i,n}^{\tau_{i}}:=\sigma\left(\varepsilon_{\ell h_{n}}:\ell \le jp_{\tau_{i},n}+i-1\right)$, $\mathcal{A}_{j,n}^{\tau_{i}}:=\mathcal{A}_{j,0,n}^{\tau_{i}}$, $\mathcal{H}_{j,i,n}^{\tau_{i}}:=\mathcal{G}_{j,i,n}^{\tau_{i}}\vee \mathcal{A}_{j,i,n}^{\tau_{i}}$
	and $\mathcal{H}_{j,n}^{\tau_{i}}:=\mathcal{H}_{j,0,n}^{\tau_{i}}$.
\end{itemize}
With respect to $X_{t}$, we assume the following conditions.
\begin{itemize}
\item[{[A1]}]
\begin{itemize}
\item[(i)] $\inf_{x,\alpha}\det A\left(x,\alpha\right)>0$.
\item[(ii)] For a constant $C$, for all $x_{1},x_{2}\in\R^{d}$,
\begin{align*}
	\sup_{\alpha\in\Theta_{1}}\left\|a\left(x_{1},\alpha\right)-a\left(x_{2},\alpha\right)\right\|+
	\sup_{\beta\in\Theta_{2}}\left|b\left(x_{1},\beta\right)-a\left(x_{2},\beta\right)\right|\le C\left|x_{1}-x_{2}\right|.
\end{align*}
\item[(iii)] For all $p\ge0$, $\sup_{t\ge0}\mathbf{E}_{\theta^{\star}}\left[\left|X_{t}\right|^{p}\right]<\infty$.
\item[(iv)] There exists a unique invariant measure $\nu=\nu_{\theta^{\star}}$ on $\left(\R^{d},\mathcal{B}\left(\R^{d}\right)\right)$ and for all $p\ge1$ and $f\in L^{p}\left(\nu\right)$ with polynomial growth,
\begin{align*}
\frac{1}{T}\int_{0}^{T}f\left(X_{t}\right)\mathrm{d}t\to^{P}\int_{\R^{d}}f\left(x\right)\nu\left(\mathrm{d}x\right).
\end{align*}
\item[(v)] For any polynomial growth function $g:\R^{d}\to\R$  satisfying $\int_{R^{d}}g\left(x\right)\nu\left(\mathrm{d}x\right) = 0$, there exist $G(x)$, $\partial_{x^{\left(i\right)}}G(x)$
with at most polynomial growth for $i=1,\ldots,d$ such that for all $x\in\R^{d}$,
\begin{align*}
L_{\theta^{\star}}G\left(x\right)=-g\left(x\right),
\end{align*}
where $L_{\theta^{\star}}$ is the infinitesimal generator of $X_{t}$.
\end{itemize}
\end{itemize}

\begin{remark}
\cite{Paradoux-Veretennikov-2001} show a sufficient condition for [A1]-(v). \cite{Uchida-Yoshida-2012} also introduce the sufficient condition for [A1]-(iii)--(v) assuming [A1]-(i)--(ii), $\sup_{x,\alpha}A\left(x,\alpha\right)<\infty$ and $^\exists c_{0}>0$, $M_{0}>0$ and $\gamma\ge 0$ such that for all $\beta\in\Theta_{2}$ and $x\in\R^{d}$ satisfying $\left|x\right|\ge M_{0}$,
\begin{align*}
	\frac{1}{\left|x\right|}x^{T}b\left(x,\beta\right)&\le -c_{0}\left|x\right|^{\gamma}.
\end{align*}
\end{remark}

\begin{itemize}
\item[{[A2]}] There exists $C>0$ such that $a:\R^{d}\times \Theta_{1}\to \R^{d}\otimes \R^{r}$ and $b:\R^{d}\times \Theta_{2}\to \R^{d}$ have continuous derivatives satisfying
\begin{align*}
\sup_{\alpha\in\Theta_{1}}\left|\partial_{x}^{j}\partial_{\alpha}^{i}a\left(x,\alpha\right)\right|&\le
C\left(1+\left|x\right|\right)^{C},\ 0\le i\le 4,\ 0\le j\le 2,\\
\sup_{\beta\in\Theta_{2}}\left|\partial_{x}^{j}\partial_{\beta}^{i}b\left(x,\beta\right)\right|&\le C\left(1+\left|x\right|\right)^{C},\ 0\le i\le 4,\ 0\le j\le 2.
\end{align*}
\end{itemize}

With the invariant measure $\nu$, we define
{\small\begin{align*}
\mathbb{Y}_{1}^{\tau_{3}} \left(\alpha;\vartheta^{\star}\right)&:=-\frac{1}{2}\int \left\{\mathrm{tr}\left(A^{\tau_{3}}\left(x,\alpha,\Lambda^{\star}\right)^{-1}A^{\tau_{3}}\left(x,\alpha^{\star},\Lambda^{\star}\right)-I_{d}\right)+\log\frac{\det A^{\tau_{3}}\left(x,\alpha,\Lambda^{\star}\right)}{\det A^{\tau_{3}}\left(x,\alpha^{\star},\Lambda^{\star}\right)}\right\}\nu\left(\mathrm{d}x\right),\\
\mathbb{Y}_{2}\left(\beta;\vartheta^{\star}\right)&:=-\frac{1}{2}
\int A\left(x,\alpha^{\star}\right)^{-1}\left[\left(b\left(x,\beta\right)-b\left(x,\beta^{\star}\right)\right)^{\otimes 2}\right]\nu\left(\mathrm{d}x\right),\\
\mathbb{V}_{1} \left(\alpha;\vartheta^{\ast}\right)
	&:=-\frac{2}{9}\int_{R^{d}}
	\left\|
A^{\tau_{1}}\left(x,\alpha,\Lambda_{\star}\right)
	-
A^{\tau_{1}}\left(x,\alpha^{\star},\Lambda_{\star}\right)\right\|^{2}\nu \left(\mathrm{d}x\right)
%,\\
	=-\frac{2}{9}\int_{R^{d}}
	\left\|
A \left(x,\alpha \right)
	-
A \left(x,\alpha^{\star} \right)\right\|^{2}\nu\left(\mathrm{d}x\right),\\
	\mathbb{V}_{2}\left(\beta;\vartheta^{\ast}\right)
&:=-\frac{1}{2}\int_{R^{d}}\left|b\left(x,\beta^{\star}\right)-b\left(x,\beta^{\star}\right)\right|^{2}\nu\left(\mathrm{d}x\right),
\end{align*}}
where $A^{\tau}\left(x,\alpha,\Lambda\right):=A\left(x,\alpha\right)+3\Lambda\mathbf{1}_{\left\{2\right\}}\left(\tau\right)$. For these functions, let us assume the following identifiability conditions hold.

\begin{itemize}
\item[{[A3]}] 
%For all $\tau_1 \in\left(1,2\right]$,
There exist $\chi_1\left(\alpha^{\star}\right)>0$ and $\chi_1'\left(\beta^{\star}\right)>0$
such that for all $\alpha\in\Theta_{1}$ and $\beta\in\Theta_{2}$,
$\mathbb{V}_{1}\left(\alpha;\theta^{\star}\right)\le -\chi_1\left(\theta^{\star}\right)\left|\alpha-\alpha^{\star}\right|^2$ 
and $\mathbb{V}_{2}\left(\beta;\theta^{\star}\right)\le -\chi_1'\left(\theta^{\star}\right)\left|\beta-\beta^{\star}\right|^2$.
\item[{[A4]}] 
For all $\tau_3 \in\left(1,2\right]$,
there exist $\chi_2\left(\alpha^{\star}\right)>0$ and $\chi_2'\left(\beta^{\star}\right)>0$
such that for all $\alpha\in\Theta_{1}$ and $\beta\in\Theta_{2}$,
$\mathbb{Y}_{1}^{\tau_{3}} \left(\alpha;\theta^{\star}\right)) \le -\chi_2\left(\theta^{\star}\right)\left|\alpha-\alpha^{\star}\right|^2$ 
and $\mathbb{Y}_{2}\left(\beta;\theta^{\star}\right) \le -\chi_2'\left(\theta^{\star}\right)\left|\beta-\beta^{\star}\right|^2$.
\end{itemize}
The next assumption is 
concerned with 
%with respect to 
the moments of noise.
\begin{itemize}
\item[{[A5]}] For any $k > 0$, $\varepsilon_{ih_{n}}$ has $k$-th moment and the components of $\varepsilon_{ih_{n}}$ are independent of the other
components for all $i$, $\left\{w_{t}\right\}_{t\ge0}$ and $x_{0}$. In addition, for all odd integer $k$, $i=0,\ldots,n$, $n\in\mathbf{N}$, and $\ell=1,\ldots,d$, $\mathbf{E}_{\theta^{\star}}\left[\left(\varepsilon_{ih_{n}}^{\left(\ell\right)}\right)^{k}\right]=0$, and $\mathbf{E}_{\theta^{\star}}\left[\varepsilon_{ih_{n}}^{\otimes 2}\right]=I_{d}$.
\end{itemize}

\begin{itemize}
\item[{[A6]}] There exist $\gamma\in\left(2/3,1\right)$ and $\gamma' \in\left(0,\gamma\right]$ such that $n^{-\gamma}\le h_{n} \le n^{-\gamma'}$ for sufficiently large $n$.
\end{itemize}

\begin{remark}\label{remarkGamma}
    $\gamma'$ should be smaller than or equal to $\gamma$ such that $n^{-\gamma}\le n^{-\gamma'}$. $\gamma$ should be larger than $2/3$ and smaller than $1$; otherwise for some $C_{1},C_{2}>0$, $k_{\tau,n}\Delta_{\tau,n}^{2}=np_{\tau,n}h_{n}^{2}=nh_{n}^{2-1/\tau}\ge n^{1-\gamma(2-1/\tau)}\ge n^{1-3\gamma/2} > C_{1} > 0$, which must converge to 0 for $\tau=\tau_{3}$ in Theorem 2, or $T_{n}=nh_{n}\le n^{1-\gamma'}<C_{2}$ as $\gamma\ge \gamma'$, which must diverge in entire discussion.
    In addition, note that under [A6], $k_{n}=nh_{n}^{1/\tau_{i}}\ge n^{1-\gamma/\tau_{i}}\ge n^{1-\gamma}\to\infty$ for all $i=1,2,3$.
\end{remark}

\section{Multi-step estimator and PLDI}

\subsection{Setting of the initial and multi-step estimators}

We define sequences of local means such that  $\left\{\bar{Y}_{\tau,j}\right\}_{j=0,\ldots,k_{\tau,n}-1}$, $\left\{\bar{X}_{\tau,j}\right\}_{j=0,\ldots,k_{\tau,n}-1}$ and $\left\{\bar{\varepsilon}_{\tau,j}\right\}_{j=0,\ldots,k_{\tau,n}-1}$, where
\begin{align*}
	\bar{Y}_{\tau,j}&=\frac{1}{p_{\tau,n}}\sum_{i=0}^{p_{\tau,n}-1}Y_{j\Delta_{\tau,n}+ih_{n}},&
	\bar{X}_{\tau,j}&=\frac{1}{p_{\tau,n}}\sum_{i=0}^{p_{\tau,n}-1}X_{j\Delta_{\tau,n}+ih_{n}},&
	\bar{\varepsilon}_{\tau,j}&=\frac{1}{p_{\tau,n}}\sum_{i=0}^{p_{\tau,n}-1}\varepsilon_{j\Delta_{\tau,n}+ih_{n}}
\end{align*}
for $j=0,\ldots,k_{\tau,n}-1$ and $\tau=\tau_{1},\tau_{2},\tau_{3}$. For the detailed properties of local means, see \citet{Favetto-2014,Favetto-2016,Nakakita-Uchida-2017,Nakakita-Uchida-2018b,Nakakita-Uchida-2018c}.

We set for $i=1,2$, $\eta_{i}\in\left(\gamma,1\right]$ and $\underline{n}_{\eta_{i}}$ such that $\underline{n}_{\eta_{i}}=n^{\eta_{i}}\le n$ satisfying $\underline{T}_{\eta_{i},n}:=\underline{n}_{\eta_{i}}h_{n}\le T_{n}:=nh_{n}$ (then it holds $\underline{T}_{\eta_{i},n}\to \infty$), and correspondingly $\underline{k}_{\eta_{i},\tau_{i},n}:=\underline{n}_{\eta_{i}}/p_{\tau_{i},n}=n^{\eta_{i}}h_{n}^{1/\tau_{i}}$.

\begin{remark}\label{remarkEta}
%With respect to $\eta_{1}$ and $\eta_{2}$, they 
$\eta_{1}$ and $\eta_{2}$
should be larger than $\gamma$ to support the divergence $\underline{T}_{\eta_{i},n}:=n^{\eta_{i}}h_{n}\to\infty$. 
%They 
$\eta_{1}$ and $\eta_{2}$
actually work to make $\left(\eta_{1}-\gamma/\tau_{1}\right)/\left(1-\gamma'/\tau_{3}\right)>0$ and $\left(\eta_{2}-\gamma\right)/\left(1-\gamma'\right)>0$, which are the quantities appearing in Remark \ref{remarkLowerBound}.
\end{remark}

Let us set $q_{1},q_{2}\in\left(0,1/2\right]$ and the following quasi likelihood functions:
\begin{align*}
	\mathbb{W}_{1,\tau_{1},n}\left(\alpha|\Lambda\right)&:=-\frac{1}{2}\sum_{j=1}^{\underline{k}_{\eta_{1},\tau_{1},n}-2}\left\|\Delta_{\tau_{1},n}^{-1}\left(\bar{Y}_{\tau_{1},j+1}-\bar{Y}_{\tau_{1},j}\right)^{\otimes2}-\frac{2}{3}A_{\tau_{1},n}\left(\bar{Y}_{\tau_{1},j-1},\alpha,\Lambda\right)\right\|^{2},\\
	\mathbb{W}_{2,\tau_{2},n}\left(\beta\right)&:=-\frac{1}{2}\sum_{j=1}^{\underline{k}_{\eta_{2},\tau_{2},n}-2}\Delta_{\tau_{2},n}^{-1}\left|\bar{Y}_{\tau_{2},j+1}-\bar{Y}_{\tau_{2},j}-\Delta_{\tau_{2},n}b\left(\bar{Y}_{\tau_{2},j-1},\beta\right)\right|^{2}.
\end{align*}
Using these quasi likelihood functions, we also define the next two functions such that
\begin{align*}
	\mathbb{H}_{1,\tau_{1},n}^{\left(0\right)}\left(\alpha|\Lambda\right)&=\frac{1}{\underline{k}_{\eta_{1},\tau_{1},n}^{1-2q_{1}}}\mathbb{W}_{1,\tau_{1},n}\left(\alpha|\Lambda\right),&
	\mathbb{H}_{2,\tau_{2},n}^{\left(0\right)}\left(\beta\right)&=\frac{1}{\underline{T}_{\eta_{2},n}^{1-2q_{2}}}\mathbb{W}_{2,\tau_{2},n}\left(\beta\right).
\end{align*}
Then, the initial estimators are defined as follows:
\begin{align*}
	\hat{\Lambda}_{n}&:=\frac{1}{2n}\sum_{i=0}^{n-1}\left(Y_{\left(i+1\right)h_{n}}-Y_{ih_{n}}\right)^{\otimes2},\\
	\tilde{\alpha}_{q_{1},\tau_{1},n}^{\left(0\right)}&:=\frac{\int_{\Theta_{1}}\alpha\exp\left(\mathbb{H}_{1,\tau_{1},n}^{\left(0\right)}\left(\alpha|\hat{\Lambda}_{n}\right)\right)\pi_{1}\left(\alpha\right)\mathrm{d}\alpha}{\int_{\Theta_{1}}\exp\left(\mathbb{H}_{1,\tau_{1},n}^{\left(0\right)}\left(\alpha|\hat{\Lambda}_{n}\right)\right)\pi_{1}\left(\alpha\right)\mathrm{d}\alpha},\\
	\tilde{\beta}_{q_{2},\tau_{2},n}^{\left(0\right)}&:=\frac{\int_{\Theta_{2}}\beta\exp\left(\mathbb{H}_{2,\tau_{2},n}^{\left(0\right)}\left(\beta\right)\right)\pi_{2}\left(\beta\right)\mathrm{d}\beta}{\int_{\Theta_{2}}\exp\left(\mathbb{H}_{2,\tau_{2},n}^{\left(0\right)}\left(\beta\right)\right)\pi_{2}\left(\beta\right)\mathrm{d}\beta},
\end{align*}
where {$0<\inf_{\alpha\in\Theta_{1}}\pi_{1}\left(\alpha\right)\le\sup_{\alpha\in\Theta_{1}}\pi_{1}\left(\alpha\right)<\infty$ and $0<\inf_{\beta\in\Theta_{2}}\pi_{2}\left(\beta\right)\le\sup_{\beta\in\Theta_{2}}\pi_{2}\left(\beta\right)<\infty$}. Note that $\hat{\Lambda}_{n}$ uses the whole data.
% in comparison to the others. 

In the next place, we define the hybrid multi-step estimators. We introduce the quasi likelihood functions in \citet{Nakakita-Uchida-2018c} such that
{\small\begin{align*}
&\mathbb{H}_{1,n}\left(\alpha|\Lambda\right)
:=-\frac{1}{2}\sum_{j=1}^{k_{\tau_{3},n}-2}
\left(\left(\frac{2}{3}\Delta_{\tau_{3},n}A_{n}^{\tau_{3}}\left(\bar{Y}_{\tau_{3},j-1},\alpha,\Lambda\right)\right)^{-1}\left[\left(\bar{Y}_{\tau_{3},j+1}-\bar{Y}_{\tau_{3},j}\right)^{\otimes 2}\right]\right.\\
&\hspace{5cm}\left.+\log\det A_{n}^{\tau_{3}}\left(\bar{Y}_{\tau_{3},j-1},\alpha,\Lambda\right)\right),\\
&\mathbb{H}_{2,n}\left(\beta|\alpha\right)
:=-\frac{1}{2}\sum_{j=1}^{k_{\tau_{3},n}-2}
\left(\left(\Delta_{\tau_{3},n}A\left(\bar{Y}_{\tau_{3},j-1},\alpha\right)\right)^{-1}\left[\left(\bar{Y}_{\tau_{3},j+1}-\bar{Y}_{\tau_{3},j}-\Delta_{\tau_{3},n}b\left(\bar{Y}_{\tau_{3},j-1},\beta\right)\right)^{\otimes 2}\right]\right),
\end{align*}} where $A_{n}^{\tau_{3}}\left(x,\alpha,\Lambda\right):=A\left(x,\alpha\right)+3\Delta_{n}^{\frac{2-\tau_{3}}{\tau_{3}-1}}\Lambda$. 
As \citet{Kamatani-Uchida-2015} and \citet{Kamatani-et-al-2016}, let us denote
\begin{align*}
    J_{1,n}\left(\alpha\right) &:= \frac{1}{k_{\tau_{3},n}}\partial_{\alpha}^{2}\mathbb{H}_{1,\tau_{3},n}\left(\alpha|\hat{\Lambda}_{n}\right), & J_{1,n}\left(\beta\right) &:= \frac{1}{T_{n}}\partial_{\beta}^{2}\mathbb{H}_{2,\tau_{3},n}\left(\beta|\hat{\alpha}_{J_{1},n}\right),\\
    K_{1,n}\left(\alpha\right) &:= \left\{J_{1,n}\left(\alpha\right)\text{ is invertible.}\right\}, & K_{2,n}\left(\beta\right) &:= \left\{J_{2,n}\left(\beta\right)\text{ is invertible.}\right\},\\
    \overline{J}_{1,n}\left(\alpha\right) &:= J_{1,n}\left(\alpha\right)\mathbf{1}_{K_{1,n}\left(\alpha\right)}+I_{m_{1}}\mathbf{1}_{K_{1,n}^{c}\left(\alpha\right)},
    &  \overline{J}_{2,n}\left(\beta\right) &:=J_{2,n}\left(\beta\right)\mathbf{1}_{K_{2,n}\left(\beta\right)}+I_{m_{2}}\mathbf{1}_{K_{2,n}^{c}\left(\beta\right)}, 
\end{align*}such that for all $k=1,\ldots,J_{1}$, $J_{1}:=\lfloor -\log_{2}\left( q_{1}\left(\eta_{1}-\gamma/\tau_{1}\right)/\left(1-\gamma'/\tau_{3}\right)\right)\rfloor$, and $\hat{\alpha}_{0,n}:=\tilde{\alpha}_{q_{1},\tau_{1},n}^{\left(0\right)}$
\begin{align*}
    \hat{\alpha}_{k,n}&:= \hat{\alpha}_{k-1,n} - \overline{J}_{1,n}^{-1}\left(\hat{\alpha}_{k-1,n}\right)\frac{1}{k_{\tau_{3},n}}\partial_{\alpha}\mathbb{H}_{1,\tau_{3},n}\left(\hat{\alpha}_{k-1,n}|\hat{\Lambda}_{n}\right),
\end{align*}
and for all $k=1,\ldots,J_{2}$, 
$J_{2}=\lfloor -\log_{2} \left( q_{2}\left(\eta_{2}-\gamma\right)/\left(1-\gamma'\right) \right)\rfloor$, 
and $\hat{\beta}_{0,n}:=\tilde{\beta}_{q_{2},\tau_{2},n}^{\left(0\right)}$,
\begin{align*}
    \hat{\beta}_{k,n}&:= \hat{\beta}_{k-1,n} - \overline{J}_{2,n}^{-1}\left(\hat{\beta}_{k-1,n}\right)\frac{1}{T_{n}}\partial_{\beta}\mathbb{H}_{2,\tau_{3},n}\left(\hat{\beta}_{k-1,n}|\hat{\alpha}_{J_{1},n}\right).
\end{align*}

\subsection{PLDIs for the quasi likelihood functions}

To examine the asymptotic behaviours of $\hat{\alpha}_{J_{1}}$ and $\hat{\beta}_{J_{2},n}$, firstly we will see that the $L^{p}$-boundedness such that
\begin{align*}
	\sup_{n\in\mathbf{N}}\mathbf{E}\left[\left|\underline{k}_{\eta_{1},\tau_{1},n}^{q_{1}}\left(\tilde{\alpha}_{q_{1},\tau_{1},n}^{\left(0\right)}-\alpha^{\star}\right)\right|^{M}\right]+\sup_{n\in\mathbf{N}}\mathbf{E}\left[\left|\underline{T}_{\eta_{2},n}^{q_{2}}\left(\tilde{\beta}_{q_{2},\tau_{2},n}^{\left(0\right)}-\beta^{\star}\right)\right|^{M}\right]<\infty.
\end{align*}
To show these boundedness, we define some random quantities: random fields such that
\begin{align*}
	&\mathbb{V}_{1,\tau_{1},n}\left(\alpha;\vartheta^{\ast}\right)\\
	&:=\frac{1}{\underline{k}_{\eta_{1},\tau_{1},n}}
	\left(\mathbb{W}_{1,\tau_{1},n}\left(\alpha|\hat{\Lambda}_{n}\right)-\mathbb{W}_{1,\tau_{1},n}\left(\alpha^{\star}|\hat{\Lambda}_{n}\right)\right)\\
	&=-\frac{1}{2\underline{k}_{\eta_{1},\tau_{1},n}}\sum_{j=1}^{\underline{k}_{\eta_{1},\tau_{1},n}-2}
	\left(\left\|\frac{2}{3}
	A_{\tau_{1},n}\left(\bar{Y}_{\tau_{1},j-1},\alpha,\hat{\Lambda}_{n}\right)\right\|^{2}
	-\left\|\frac{2}{3}
	A_{\tau_{1},n}\left(\bar{Y}_{\tau_{1},j-1},\alpha^{\star},\hat{\Lambda}_{n}\right)\right\|^{2}\right.\\
	&\hspace{3cm}\left.-2\left(\Delta_{\tau_{1},n}^{-1}\left(\bar{Y}_{\tau_{1},j+1}-\bar{Y}_{\tau_{1},j}\right)^{\otimes2}\right)\left[\frac{2}{3}A_{\tau_{1},n}\left(\bar{Y}_{\tau_{1},j-1},\alpha,\hat{\Lambda}_{n}\right)\right]\right.\\
	&\hspace{3cm}\left.+2\left(\Delta_{\tau_{1},n}^{-1}\left(\bar{Y}_{\tau_{1},j+1}-\bar{Y}_{\tau_{1},j}\right)^{\otimes2}\right)\left[\frac{2}{3}A_{\tau_{1},n}\left(\bar{Y}_{\tau_{1},j-1},\alpha^{\star},\hat{\Lambda}_{n}\right)\right]\right),\\
	&\mathbb{V}_{2,\tau_{2},n}\left(\beta;\vartheta^{\star}\right)\\
	&:=\frac{1}{T_{n}}\left(\mathbb{W}_{2,\tau_{2},n}\left(\beta\right)-\mathbb{W}_{2,\tau_{2},n}\left(\beta^{\star}\right)\right)\\
	&=-\frac{1}{2\underline{k}_{\eta_{2},\tau_{2},n}}\sum_{j=1}^{\underline{k}_{\eta_{2},\tau_{2},n}-2}\left(\left|b\left(\bar{Y}_{\tau_{2},j-1},\beta\right)\right|^{2}-\left|b\left(\bar{Y}_{\tau_{2},j-1},\beta^{\star}\right)\right|^{2}\right.\\
	&\hspace{3.5cm}\left.-2\Delta_{\tau_{2},n}^{-1}\left(\bar{Y}_{\tau_{2},j+1}-\bar{Y}_{\tau_{2},j}\right)\left[b\left(\bar{Y}_{\tau_{2},j-1},\beta\right)\right]\right.\\
	&\hspace{3.5cm}\left.+2\Delta_{\tau_{2},n}^{-1}\left(\bar{Y}_{\tau_{2},j+1}-\bar{Y}_{\tau_{2},j}\right)\left[b\left(\bar{Y}_{\tau_{2},j-1},\beta^{\star}\right)\right]\right);
\end{align*}
score functions such that
\begin{align*}
	S_{1,\tau_{1},n}\left(\vartheta^{\star}\right)&:=-\frac{2}{3\underline{k}_{\eta_{1},\tau_{1},n}^{1-q_{1}}}\sum_{j=1}^{\underline{k}_{\eta_{1},\tau_{1},n}-2}\left(\partial_{\alpha}A\left(\bar{Y}_{j-1},\alpha^{\star}\right)\right)\\
	&\hspace{4cm}\left[\Delta_{\tau_{1},n}^{-1}\left(\bar{Y}_{\tau_{1},j+1}-
	\bar{Y}_{\tau_{1},j}\right)^{\otimes2}-\frac{2}{3}A_{\tau_{1},n}\left(\bar{Y}_{\tau_{1},j-1},\alpha^{\star},\hat{\Lambda}_{n}\right)\right],\\
	S_{2,\tau_{2},n}\left(\vartheta^{\star}\right)&:=-\frac{1}{\underline{T}_{\eta_{1},n}^{1-q_{2}}}\sum_{j=1}^{\underline{k}_{\eta_{2},\tau_{2},n}-2}\left(\partial_{\beta}b\left(\bar{Y}_{\tau_{2},j-1},\beta^{\star}\right)\right)
	\left[\left(\bar{Y}_{\tau_{2},j+1}-\bar{Y}_{\tau_{2},j}\right)-\Delta_{\tau_{2},n}b\left(\bar{Y}_{\tau_{2},j-1},\beta^{\star}\right)\right];
\end{align*}
the observed information matrices such that
\begin{align*}
	&\Gamma_{1,\tau_{1},n}\left(\alpha;\vartheta^{\star}\right)\left[u_{1}^{\otimes2}\right]\\
	&:=-\frac{2}{3\underline{k}_{\eta_{1},\tau_{1},n}}\sum_{j=1}^{\underline{k}_{\eta_{1},\tau_{1},n}-2}\left(\partial_{\alpha}^{2}A\left(\bar{Y}_{j-1},\alpha\right)\right)\\
	&\hspace{4cm}\left[u_{1}^{\otimes2},\Delta_{\tau_{1},n}^{-1}\left(\bar{Y}_{\tau_{1},j+1}-\bar{Y}_{\tau_{1},j}\right)^{\otimes2}-\frac{2}{3}A_{\tau_{1},n}\left(\bar{Y}_{\tau_{1},j-1},\alpha,\hat{\Lambda}_{n}\right)\right]\\
	&\hspace{1cm}+\frac{4}{9\underline{k}_{\eta_{1},\tau_{1},n}}\sum_{j=1}^{\underline{k}_{\eta_{1},\tau_{1},n}-2}\left(\partial_{\alpha}A\left(\bar{Y}_{\tau_{1},j-1},\alpha\right)\right)[u_{1}^{\otimes2},\partial_{\alpha}A\left(\bar{Y}_{\tau_{1},j-1},\alpha\right)],\\
	&\Gamma_{2,\tau_{2},n}\left(\beta;\vartheta^{\star}\right)\left[u_{2}^{\otimes2}\right]\\
	&:=-\frac{1}{\underline{k}_{\eta_{2},\tau_{2},n}}\sum_{j=1}^{\underline{k}_{\eta_{2},\tau_{2},n}-2}\left(\partial_{\beta}^{2}b\left(\bar{Y}_{j-1},\beta\right)\right)
	\left[u_{2}^{\otimes2},\Delta_{\tau_{2},n}^{-1}\left(\bar{Y}_{\tau_{2},j+1}-\bar{Y}_{\tau_{2},j}\right)-b\left(\bar{Y}_{\tau_{2},j-1},\beta\right)\right]\\
	&\hspace{1cm}+\frac{1}{\underline{k}_{\eta_{2},\tau_{2},n}}\sum_{j=1}^{\underline{k}_{\eta_{2},\tau_{2},n}-2}\left(\partial_{\beta}b\left(\bar{Y}_{\tau_{2},j-1},\beta\right)\right)
	\left[u_{2}^{\otimes2},\partial_{\beta}b\left(\bar{Y}_{\tau_{2},j-1},\beta\right)\right];
\end{align*}
and the limiting information matrices such that
\begin{align*}
	\Gamma_{1,\tau_{1}}\left(\vartheta^{\star}\right)\left[u_{1}^{\otimes2}\right]&:=\frac{4}{9}\int_{\mathbf{R}^{d}}\left(\partial_{\alpha}A\left(x,\alpha^{\star}\right)\right)[u_{1}^{\otimes2},\partial_{\alpha}A\left(x,\alpha^{\star}\right)]\nu_{0}\left(\mathrm{d}x\right),\\
	\Gamma_{2}\left(\vartheta^{\star}\right)\left[u_{2}^{\otimes2}\right]&:=\int_{\mathbf{R}^{d}}\left(\partial_{\beta}b\left(x,\beta^{\star}\right)\right)
	\left[u_{2}^{\otimes2},\partial_{\beta}b\left(x,\beta^{\star}\right)\right]\nu_{0}\left(\mathrm{d}x\right).
\end{align*}

\begin{lemma}\label{initAlphaLemma} 
Assume [A1]-[A2] and [A5]-[A6].
Moreover, assume $\underline{k}_{\eta_{1},\tau_{1},n}^{q_{1}}\Delta_{\tau_{1},n}\to0$.
\begin{enumerate}
\item For every $p>1$,
\begin{align*}
	\sup_{n\in\mathbf{N}}\mathbf{E}\left[\left|S_{1,\tau_{1},n}\left(\vartheta^{\star}\right)\right|^{p}\right]<\infty.
\end{align*}
\item Let $\epsilon_{1}=\epsilon_{0}/2$. Then for every $p>0$,
\begin{align*}
	\sup_{n\in\mathbf{N}}\mathbf{E}\left[\left(\sup_{\alpha\in\Theta_{1}}\underline{k}_{\eta_{1},\tau_{1},n}^{\epsilon_{1}}\left|\mathbb{V}_{1,\tau_{1},n}\left(\alpha;\vartheta^{\star}\right)-\mathbb{V}_{1,\tau_{1}}\left(\alpha;\vartheta^{\star}\right)\right|\right)^{p}\right]<\infty.
\end{align*}
\item For any $M_{3}>0$,
\begin{align*}
\sup_{n\in\mathbf{N}}\mathbf{E}_{\theta^{\star}}\left[\left(\underline{k}_{\eta_{1},\tau_{1},n}^{-1}\sup_{\vartheta\in\Xi}
\left|\partial_{\alpha}^{3}\mathbb{W}_{1,\tau_{1},n}\left(\alpha;\Lambda\right)\right|\right)^{M_{3}}\right]<\infty.
\end{align*}
\item Let $\epsilon_{1}=\epsilon_{0}/2$. Then for $M_{4}>0$,
\begin{align*}
\sup_{n\in\mathbf{N}}\mathbf{E}_{\theta^{\star}}\left[\left(\underline{k}_{\eta_{1},\tau_{1},n}^{\epsilon_{1}}\left|\Gamma_{1,\tau_{1},n}\left(\alpha^{\star};\vartheta^{\star}\right)-\Gamma_{1,\tau_{1}}\left(\vartheta^{\star}\right)\right|\right)^{M_{4}}\right]<\infty.
\end{align*}
\end{enumerate}
\end{lemma}

\begin{remark}
	Let us assume $h_{n}:=n^{-7/10}$, $\tau_{1}=2$; then $T_{n}=n^{3/10}$, $p_{\tau_{1},n}=h_{n}^{-1/2}=n^{7/20}$, $\Delta_{\tau_{1},n}=p_{\tau_{1},n}h_{n}=n^{-7/20}$ and $k_{\tau_{1},n}=n^{13/20}$. If we set $\eta_{1}=47/60$, then $\underline{k}_{\eta_{1},\tau_{1},n}:=n^{13/30}$, $\underline{k}_{\eta_{1},\tau_{1},n}\Delta_{\tau_{1},n}=n^{1/12}\to \infty $ and for all $q_{1}\in\left(0,\frac{1}{2}\right]$, $\underline{k}_{\eta_{1},\tau_{1},n}^{q_{1}}\Delta_{\tau_{1},n}=n^{13q_{1}/30-7/20}\to0$ as $n\to\infty$.
	
	Now we 
	%consider the concrete number 
	give an example of $\underline{k}_{\eta_{1},\tau_{1},n}$ which directly affects the burden of derivation of $\tilde{\alpha}_{q_{1},\tau_{1},n}^{\left(0\right)}$. If we have $n=10^{8}$, then $\underline{k}_{\eta_{1},\tau_{1},n}=10^{52/15}
	\approx2929$, while $k_{\tau_{1},n}=10^{26/5}\approx158489$.
	
	With respect to $J_{1}$, as $\gamma=\gamma'=7/10$, when we set $\tau_{3}=1.9$ and $q_{1}=1/4$, then $J_{1}=\lfloor -\log_{2}\left( q_{1}\left(\eta_{1}-\gamma/\tau_{1}\right)/\left(1-\gamma'/\tau_{3}\right)\right)\rfloor=\lfloor-\log_{2}\left(0.25\left(47/60-7/20\right)/\left(1-7/19\right)\right)\rfloor = 2$.
\end{remark}

\begin{lemma}\label{initBetaLemma}  
Assume [A1]-[A2] and [A5]-[A6].
Moreover, assume $\underline{k}_{\eta_{2},\tau_{2},n}^{q_{2}}\Delta_{\tau_{2},n}^{1+q_{2}}\to0$.
\begin{enumerate}
	\item For every $p>1$,
	\begin{align*}
	\sup_{n\in\mathbf{N}}\mathbf{E}\left[\left|S_{2,\tau_{2},n}\left(\vartheta^{\star}\right)\right|^{p}\right]<\infty.
	\end{align*}
	\item Let $\epsilon_{1}=\epsilon_{0}/2$. Then for every $p>0$,
	\begin{align*}
	\sup_{n\in\mathbf{N}}\mathbf{E}\left[\left(\sup_{\beta\in\Theta_{2}}\underline{T}_{\eta_{2},n}^{\epsilon_{1}}\left|\mathbb{V}_{2,\tau_{2},n}\left(\beta;\vartheta^{\star}\right)-\mathbb{V}_{2}\left(\beta;\vartheta^{\star}\right)\right|\right)^{p}\right]<\infty.
	\end{align*}
	\item For every $M_{3}>0$,
	\begin{align*}
	\sup_{n\in\mathbf{N}}\mathbf{E}\left[\left( \underline{T}_{\eta_{2},n}^{-1}\sup_{\beta\in\Theta_{2}}\left|\partial_{\beta}^{3}\mathbb{W}_{2,\tau_{2},n}\left(\beta\right)\right|\right)^{M_{3}}\right]&<\infty.
	\end{align*}
	\item Let $\epsilon_{1}=\epsilon_{0}/2$. Then for every $M_{4}>0$,
	\begin{align*}
	\sup_{n\in\mathbf{N}}\mathbf{E}\left[\left(\underline{T}_{\eta_{2},n}^{\epsilon_{1}}\left|\Gamma_{2,\tau_{2},n}\left(\beta^{\star};\vartheta^{\star}\right)-\Gamma_{2}\left(\vartheta^{\star}\right)\right|\right)^{M_{4}}\right]&<\infty.
	\end{align*}
\end{enumerate}
\end{lemma}

\begin{remark}
	Let us assume $h_{n}:=n^{-7/10}$, $\tau_{2}=1.2$; then $T_{n}=n^{3/10}$, $p_{\tau_{2},n}=h_{n}^{-5/6}=n^{7/12}$, $\Delta_{\tau_{2},n}=p_{\tau_{2},n}h_{n}=n^{-7/60}$, $\eta_{2}=5/6$ and $\underline{k}_{\eta_{2},\tau_{2},n}=n^{1/4}$. It holds for all $q_{2}\in\left(0,\frac{1}{2}\right]$, 
	$\underline{k}_{\eta_{2},\tau_{2},n}^{q_{2}}\Delta_{\tau_{2},n}^{1+q_{2}}=n^{q_{2}/4-7\left(1+q_{2}\right)/60}=n^{2q_{2}/15-7/60}\to0$.
	
%	Now we consider the concrete number of 
	A concrete example of $\underline{k}_{\tau_{2},n}$, which directly affects the burden of derivation of $\tilde{\beta}_{q_{2},\tau_{2},n}^{\left(0\right)}$, is given as follows. If we have $n=10^{8}$, then $\underline{k}_{\tau_{2},n}=10^{2}$, while $k_{\tau,n}=10^{26/5}\approx158489$ for $\tau=2$. With respect to $J_{2}$, if $q_{2}=2^{-8}$, we have $J_{2}=\lfloor-\log_{2}\left( q_{2}\left(\eta_{2}-\gamma\right)/\left(1-\gamma'\right) \right)\rfloor=9$.
\end{remark}

In addition to these evaluations, we define the sets for all $r>0$, 
\begin{align*}
	\mathbb{U}_{1,q_{1},n}^{\left(0\right)}\left(\alpha^{\star}\right)&:=\left\{u_{1}\in\mathbf{R}^{m_{1}};\alpha^{\star}+\underline{k}_{\eta_{1},\tau_{1},n}^{-q_{1}}u_{1}\in\Theta_{1}\right\},\\
	V_{1,q_{1},n}^{\left(0\right)}\left(r\right)&:=\left\{u_{1}\in \mathbb{U}_{1,q_{1},n}^{\left(0\right)}\left(\alpha^{\star}\right); r\le \left|u_{1}\right|\right\},\\
	\mathbb{U}_{2,q_{2},n}^{\left(0\right)}\left(\beta^{\star}\right)&:=\left\{u_{2}\in\mathbf{R}^{m_{2}};\beta^{\star}+\underline{T}_{\eta_{2},n}^{-q_{2}}u_{2}\in\Theta_{2}\right\},\\
	V_{2,q_{2},n}^{\left(0\right)}\left(r\right)&:=\left\{u_{2}\in \mathbb{U}_{2,q_{2},n}^{\left(0\right)}\left(\beta^{\star}\right); r\le \left|u_{2}\right|\right\}
\end{align*}
and the statistical random fields: 
\begin{align*}
	\mathbb{Z}_{1,\tau_{1},n}^{\left(0\right)}\left(u_{1};\vartheta^{\star}\right)&:=\exp\left(\mathbb{H}_{1,\tau_{1},n}^{\left(0\right)}\left(\alpha^{\star}+\underline{k}_{\eta_{1},\tau_{1},n}^{-q_{1}}u_{1}|\hat{\Lambda}_{n}\right)-\mathbb{H}_{1,\tau_{1},n}^{\left(0\right)}\left(\alpha^{\star}|\hat{\Lambda}_{n}\right)\right),\\
	\mathbb{Z}_{2,\tau_{2},n}^{\left(0\right)}\left(u_{2};\vartheta^{\star}\right)&:=\exp\left(\mathbb{H}_{2,\tau_{2},n}^{\left(0\right)}\left(\beta^{\star}+\underline{T}_{\eta_{2},n}^{-q_{2}}u_{2}\right)-\mathbb{H}_{2,\tau_{2},n}^{\left(0\right)}\left(\beta^{\star}\right)\right).
\end{align*}

We have the following results as for the random fields and consequently the Bayes type estimators $\tilde{\alpha}_{q_{1},\tau_{1},n}^{\left(0\right)}$ and $\tilde{\beta}_{q_{2},\tau_{2},n}^{\left(0\right)}$ using Lemma \ref{initAlphaLemma} and Lemma \ref{initBetaLemma}.
\begin{theorem}\label{PLDI}
Let $L>0$ and $i=1,2$. 
Assume [A1]-[A3] and [A5]-[A6].
Then, there exists $C\left(L\right)>0$ such that
\begin{align*}
P\left[\sup_{u_{i}\in V_{i,q_{i},n}\left(r\right)}\mathbb{Z}_{i,\tau_{i},n}^{\left(0\right)}\left(u_{i}\right)\ge e^{-r}\right] \le \frac{C\left(L\right)}{r^{L}}    
\end{align*}
for all $r>0$ and $n\in\mathbf{N}$. Moreover, for all $M>0$,
\begin{align*}
	\sup_{n\in\mathbf{N}}\mathbf{E}\left[\left|\underline{k}_{\eta_{1},\tau_{1},n}^{q_{1}}\left(\tilde{\alpha}_{q_{1},\tau_{1},n}^{\left(0\right)}-\alpha^{\star}\right)\right|^{M}\right]+\sup_{n\in\mathbf{N}}\mathbf{E}\left[\left|\underline{T}_{\eta_{2},n}^{q_{2}}\left(\tilde{\beta}_{q_{2},\tau_{2},n}^{\left(0\right)}-\beta^{\star}\right)\right|^{M}\right]<\infty.
\end{align*}
\end{theorem}

\begin{remark}\label{remarkLowerBound}
    Note that 
    \begin{align*}
        \underline{k}_{\eta_{1},\tau_{1},n}&=n^{\eta_{1}} h_{n}^{1/\tau_{1}}\ge n^{\eta_{1}}\left(n^{-\gamma}\right)^{1/\tau_{1}}=n^{\eta_{1}-\gamma/\tau_{1}}=n^{\left(1-\gamma'/\tau_{3}\right)\left(1-\gamma'/\tau_{3}\right)^{-1}\left(\eta_{1}-\gamma/\tau_{1}\right)}\\
        &= \left(nn^{-\gamma'/\tau_{3}}\right)^{\left(\eta_{1}-\gamma/\tau_{1}\right)/\left(1-\gamma'/\tau_{3}\right)} \ge \left(nh^{1/\tau_{3}}\right)^{\left(\eta_{1}-\gamma/\tau_{1}\right)/\left(1-\gamma'/\tau_{3}\right)} = k_{\tau_{3},n}^{\left(\eta_{1}-\gamma/\tau_{1}\right)/\left(1-\gamma'/\tau_{3}\right)}
    \end{align*}
    and
    \begin{align*}
        \underline{T}_{\eta_{2},n}&=n^{\eta_{2}}h_{n}\ge n^{\eta_{2}}n^{-\gamma}=\left(n^{1-\gamma'}\right)^{\left(\eta_{2}-\gamma\right)/\left(1-\gamma'\right)}\ge T_{n}^{\left(\eta_{2}-\gamma\right)/\left(1-\gamma'\right)};
    \end{align*}
    therefore,
    \begin{align*}
            \sup_{n\in\mathbf{N}}\mathbf{E}\left[\left|k_{\tau,n}^{q_{1}\left(\eta_{1}-\gamma/\tau_{1}\right)/\left(1-\gamma'/\tau_{3}\right)}\left(\tilde{\alpha}_{q_{1},\tau_{1},n}^{\left(0\right)}-\alpha^{\star}\right)\right|^{M}\right]\le \sup_{n\in\mathbf{N}}\mathbf{E}\left[\left|\underline{k}_{\eta_{1},\tau_{1},n}^{q_{1}}\left(\tilde{\alpha}_{q_{1},\tau_{1},n}^{\left(0\right)}-\alpha^{\star}\right)\right|^{M}\right]<\infty,
    \end{align*}
    and
    \begin{align*}
         \sup_{n\in\mathbf{N}}\mathbf{E}\left[\left|T_{n}^{q_{2}\left(\eta_{2}-\gamma\right)/\left(1-\gamma'\right)}\left(\tilde{\beta}_{q_{2},\tau_{2},n}^{\left(0\right)}-\beta^{\star}\right)\right|^{M}\right]\le \sup_{n\in\mathbf{N}}\mathbf{E}\left[\left|\underline{T}_{\eta_{2},n}^{q_{2}}\left(\tilde{\beta}_{q_{2},\tau_{2},n}^{\left(0\right)}-\beta^{\star}\right)\right|^{M}\right]<\infty.
    \end{align*}
\end{remark}

\begin{en-text}
In order to state asymptotic properties of the hybrid estimator, we give the following notation.
\begin{itemize}
	\item We define the real-valued function as for $l_{1},l_{2},l_{3},l_{4}=1,\ldots,d$:
	\begin{align*}
	&V\left((l_1,l_2),(l_3,l_4)\right)\\
	&:=\sum_{k=1}^{d}\left(\Lambda_{\star}^{1/2}\right)^{(l_1,k)}\left(\Lambda_{\star}^{1/2}\right)^{(l_2,k)}\left(\Lambda_{\star}^{1/2}\right)^{(l_3,k)}\left(\Lambda_{\star}^{1/2}\right)^{(l_4,k)}
	\left(\mathbf{E}_{\theta^{\star}}\left[\left|\epsilon_{0}^{\left(k\right)}\right|^4\right]-3\right)\\
	&\qquad+\frac{3}{2}\left(\Lambda_{\star}^{(l_1,l_3)}\Lambda_{\star}^{(l_2,l_4)}+\Lambda_{\star}^{(l_1,l_4)}\Lambda_{\star}^{(l_2,l_3)}\right),
	\end{align*}
	and with the function $\sigma$ as for $i=1,\ldots,d$ and $j=i,\ldots,d$,
	\begin{align*}
		\sigma\left(i,j\right):=\begin{cases}
		j&\text{ if }i=1,\\
		\sum_{\ell=1}^{i-1}\left(d-\ell+1\right)+j-i+1 & \text{ if }i>1,
		\end{cases}
	\end{align*}
	we define the matrix $W_{1}$ as for $i_{1},i_{2}=1,\ldots,d(d+1)/2$,
	\begin{align*}
		W_{1}^{\left(i_{1},i_{2}\right)}:=V\left(\sigma^{-1}\left(i_{1}\right),\sigma^{-1}\left(i_{2}\right)\right).
	\end{align*}
	\item Let
	\begin{align*}
	&\left\{B_{\kappa}(x)\left|\kappa=1,\ldots,m_1,\ B_{\kappa}=(B_{\kappa}^{(j_1,j_2)})_{j_1,j_2}\right.\right\},\\
	&\left\{f_{\lambda}(x)\left|\lambda=1,\ldots,m_2,\ f_{\lambda}=(f^{(1)}_{\lambda},\ldots,f^{(d)}_{\lambda})\right.\right\}
	\end{align*}
	be
	sequences of $\Re^d\otimes \Re^d$-valued functions and $\Re^d$-valued ones respectively such that the components of themselves and their derivative with respect to $x$ are polynomial growth functions for all $\kappa$ and $\lambda$. 
	Then we define the following matrix-valued functionals, for $\bar{B}_{\kappa}:=\frac{1}{2}\left(B_{\kappa}+B_{\kappa}^T\right)$,
	\begin{eqnarray*}
	&& \left(W_2^{(\tau)}\left(\left\{B_{\kappa}:\kappa=1,\ldots,m_{1}\right\}\right)\right)^{(\kappa_1,\kappa_2)}\\
	&&:= \left\{
		 \begin{array}{lc}
		\nu\left(\mathrm{tr}\left\{\left(\bar{B}_{\kappa_1}A\bar{B}_{\kappa_2}A\right)(\cdot)\right\}\right) &\text{ if }\tau\in(1,2),\\
		\nu\left(\mathrm{tr}\left\{\left(\bar{B}_{\kappa_1}A\bar{B}_{\kappa_2}A+4\bar{B}_{\kappa_1}A\bar{B}_{\kappa_2}\Lambda_{\star}+12\bar{B}_{\kappa_1}\Lambda_{\star}\bar{B}_{\kappa_2}\Lambda_{\star}\right)(\cdot)\right\}\right)
		&\text{ if }\tau=2,		
		\end{array}
		\right. \\
	&&\left(W_3(\left\{f_{\lambda}:\lambda=1,\ldots,m_{2}\right\})\right)^{(\lambda_1,\lambda_2)}\\
	&&:=
		\nu\left(\left(f_{\lambda_1}A\left(f_{\lambda_2}\right)^T\right)(\cdot)\right),
	\end{eqnarray*}
	where $\nu=\nu_{\theta^{\star}}$ is the invariant measure of $X_{t}$ discussed in the following assumption [A1]-(iv), and for all function $f$ on $\Re^{d}$, $\nu\left(f\left(\cdot\right)\right):=\int_{\Re^{d}} f\left(x\right)\nu\left(\mathrm{d}x\right)$.
	\item Let
\begin{align*}
\mathcal{I}^{\tau}\left(\vartheta^{\star}\right)&:=\mathrm{diag}\left\{W_{1}, \mathcal{I}^{(2,2),\tau},\mathcal{I}^{(3,3)}\right\}\left(\vartheta^{\star}\right), \\
\mathcal{J}^{\tau}\left(\vartheta^{\star}\right)&:=\mathrm{diag}\left\{I_{d(d+1)/2},\mathcal{J}^{(2,2),\tau},\mathcal{J}^{(3,3)}\right\}(\vartheta^{\star}).
\end{align*}
\item For $i_1,i_2\in\left\{1,\ldots,m_1\right\}$,
\begin{align*}
\mathcal{I}^{(2,2),\tau}(\vartheta^{\star})&:=
W_2^{(\tau)}\left(\left\{\frac{3}{4}\left( A^{\tau}\right)^{-1}\left(\partial_{\alpha^{(k_1)}}A\right)\left(A^{\tau}\right)^{-1}(\cdot,\vartheta^{\star}):k_{1}=1,\ldots,m_{1}\right\}\right),\\
\mathcal{J}^{(2,2),\tau}(\vartheta^{\star})&:=
\left[\frac{1}{2}\nu\left(\mathrm{tr}\left\{\left(A^{\tau}\right)^{-1}\left(\partial_{\alpha^{(i_1)}}A\right)\left(A^{\tau}\right)^{-1}
	\left(\partial_{\alpha^{(i_2)}}A\right)\right\}(\cdot,\vartheta^{\star})\right)\right]_{i_1,i_2},
\end{align*}
\item
For $j_1,j_2\in\left\{1,\ldots,m_2\right\}$,
\begin{align*}
\mathcal{I}^{(3,3)}(\theta^{\star})&=\mathcal{J}^{(3,3)}(\theta^{\star}):=
\left[\nu\left(\left(A\right)^{-1}\left[\partial_{\beta^{(j_1)}}b,\partial_{\beta^{(j_2)}}b\right](\cdot,\theta^{\star})\right)\right]_{j_1,j_2}.
\end{align*}
\item
$\hat{\theta}_{\varepsilon,n}:=\mathrm{vech}\hat{\Lambda}_{n}$ and $\theta_{\varepsilon}^{\star}:=\mathrm{vech}\Lambda^{\star}$.
\item Let
	\begin{align*}
		\left(\zeta_{0},\zeta_{1},\zeta_{2}\right)
		\sim N_{d\left(d+1\right)/2+m_{1}+m_{2}}\left(\mathbf{0},\left(\mathcal{J}^{\tau}\left(\vartheta^{\star}\right)\right)^{-1}\left(\mathcal{I}^{\tau}\left(\vartheta^{\star}\right)\right)\left(\mathcal{J}^{\tau}\left(\vartheta^{\star}\right)\right)^{-1}
		\right)
	\end{align*}
\end{itemize}
\end{en-text}

In order to state asymptotic properties of the hybrid multi-step estimators, we introduce the notation as follows.
Let $V\left((l_1,l_2),(l_3,l_4)\right)$ be the real-valued function as for $l_{1},l_{2},l_{3},l_{4}=1,\ldots,d$,
	\begin{align*}
	&V\left((l_1,l_2),(l_3,l_4)\right)\\
	&:=\sum_{k=1}^{d}\left(\Lambda_{\star}^{1/2}\right)^{(l_1,k)}\left(\Lambda_{\star}^{1/2}\right)^{(l_2,k)}\left(\Lambda_{\star}^{1/2}\right)^{(l_3,k)}\left(\Lambda_{\star}^{1/2}\right)^{(l_4,k)}
	\left(\mathbf{E}_{\theta^{\star}}\left[\left|\epsilon_{0}^{\left(k\right)}\right|^4\right]-3\right)\\
	&\qquad+\frac{3}{2}\left(\Lambda_{\star}^{(l_1,l_3)}\Lambda_{\star}^{(l_2,l_4)}+\Lambda_{\star}^{(l_1,l_4)}\Lambda_{\star}^{(l_2,l_3)}\right),
	\end{align*}
	and the function $\sigma$ is defined as for $i=1,\ldots,d$ and $j=i,\ldots,d$,
	\begin{align*}
		\sigma\left(i,j\right):=\begin{cases}
		j&\text{ if }i=1,\\
		\sum_{\ell=1}^{i-1}\left(d-\ell+1\right)+j-i+1 & \text{ if }i>1.
		\end{cases}
	\end{align*}
	Furthermore, for $i_{1},i_{2}=1,\ldots,d(d+1)/2$,
	\begin{align*}
		W_{1}^{\left(i_{1},i_{2}\right)}:=V\left(\sigma^{-1}\left(i_{1}\right),\sigma^{-1}\left(i_{2}\right)\right).
	\end{align*}
%	Let
%	\begin{align*}
%	$\left\{B_{\kappa}(x)\left|\kappa=1,\ldots,m_1, \ B_{\kappa}=(B_{\kappa}^{(j_1,j_2)})_{j_1,j_2=1,\ldots, d}\right.\right\}$ and        
%	$\left\{f_{\lambda}(x)\left|\lambda=1,\ldots,m_2,\ f_{\lambda}=(f^{(1)}_{\lambda},\ldots,f^{(d)}_{\lambda})\right.\right\}$
%	\end{align*}
Let
	$\left\{ B_{\kappa}(x)\left|\kappa=1,\ldots,m_1 \right. \right\}$ and        
	$\left\{f_{\lambda}(x)\left|\lambda=1,\ldots,m_2 \right. \right\}$
	be
	sequences of $\mathbf{R}^d\otimes \mathbf{R}^d$-valued functions and $\mathbf{R}^d$-valued ones respectively such that their components and their derivatives with respect to $x$ are polynomial growth functions for all $\kappa$ and $\lambda$. 
%	Then we define the following matrix-valued functionals, 
	For $\bar{B}_{\kappa}(x):=\frac{1}{2}\left(B_{\kappa}(x)+B_{\kappa}(x)^T\right)$,
	\begin{eqnarray*}
	&& \left(W_2^{(\tau)}\left(\left\{B_{\kappa}(x):\kappa=1,\ldots,m_{1}\right\}\right)\right)^{(\kappa_1,\kappa_2)}\\
	&&:= \left\{
		 \begin{array}{lc}
		 \int_{\mathbf{R}^d} \mathrm{tr}\left\{\left(\bar{B}_{\kappa_1}A\bar{B}_{\kappa_2}A\right)(x)\right\} \nu(dx)  &\text{ if }\tau\in(1,2),\\
		  \int_{\mathbf{R}^d}   \mathrm{tr}\left\{\left(\bar{B}_{\kappa_1}A\bar{B}_{\kappa_2}A
		+4\bar{B}_{\kappa_1}A\bar{B}_{\kappa_2}\Lambda_{\star}+12\bar{B}_{\kappa_1}\Lambda_{\star}\bar{B}_{\kappa_2}\Lambda_{\star}\right)(x)   \right\}   \nu(dx)   
		&\text{ if }\tau=2.		
		\end{array}
		\right. 
%		\\
%	&&\left(W_3(\left\{f_{\lambda}:\lambda=1,\ldots,m_{2}\right\})\right)^{(\lambda_1,\lambda_2)}\\
%	&&:=
%		\nu\left(\left(f_{\lambda_1}A\left(f_{\lambda_2}\right)^T\right)(\cdot)\right),
	\end{eqnarray*}
%	where $\nu=\nu_{\theta^{\star}}$ is the invariant measure of $X_{t}$ discussed in the following assumption [A1]-(iv), 
%    and for all function $f$ on $\Re^{d}$, $\nu\left(f\left(\cdot\right)\right):=\int_{\Re^{d}} f\left(x\right)\nu\left(\mathrm{d}x\right)$.
Set
\begin{align*}
\mathcal{I}^{\tau}\left(\vartheta^{\star}\right)&:=\mathrm{diag}\left\{W_{1}, \mathcal{I}^{(2,2),\tau},\mathcal{I}^{(3,3)}\right\}\left(\vartheta^{\star}\right), \\
\mathcal{J}^{\tau}\left(\vartheta^{\star}\right)&:=\mathrm{diag}\left\{I_{d(d+1)/2},\mathcal{J}^{(2,2),\tau},\mathcal{J}^{(3,3)}\right\}(\vartheta^{\star}), \\
%\end{align*}
%For $i_1,i_2\in\left\{1,\ldots,m_1\right\}$,
%\begin{align*}
\mathcal{I}^{(2,2),\tau}(\vartheta^{\star})&:=
W_2^{(\tau)}\left(\left\{\frac{3}{4}\left( A^{\tau}\right)^{-1}\left(\partial_{\alpha^{(k_1)}}A\right)\left(A^{\tau}\right)^{-1}(\cdot,\vartheta^{\star}):k_{1}=1,\ldots,m_{1}\right\}\right),\\
\mathcal{J}^{(2,2),\tau}(\vartheta^{\star})&:=
\left(\frac{1}{2}  \int_{\mathbf{R}^d} \mathrm{tr}\left\{\left(A^{\tau}\right)^{-1}\left(\partial_{\alpha^{(i_1)}}A\right)\left(A^{\tau}\right)^{-1}
	\left(\partial_{\alpha^{(i_2)}}A\right)\right\}(x,\vartheta^{\star}) \nu(dx)\right)_{i_1,i_2 =1,\ldots,m_1}, \\
%\end{align*}
%For $j_1,j_2\in\left\{1,\ldots,m_2\right\}$,
%\begin{align*}
\mathcal{I}^{(3,3)}(\theta^{\star})&=\mathcal{J}^{(3,3)}(\theta^{\star}):=
\left(  \int_{\mathbf{R}^d} \left(A\right)^{-1}\left[\partial_{\beta^{(j_1)}}b,\partial_{\beta^{(j_2)}}b\right](x,\theta^{\star}) \nu(dx) \right)_{j_1,j_2=1,\ldots,m_2}.
\end{align*}
Let
$\hat{\theta}_{\varepsilon,n}:=\mathrm{vech}\hat{\Lambda}_{n}$, $\theta_{\varepsilon}^{\star}:=\mathrm{vech}\Lambda^{\star}$
and 
\begin{align*}
		\left(\zeta_{0},\zeta_{1},\zeta_{2}\right)
		\sim N_{d\left(d+1\right)/2+m_{1}+m_{2}}
		\left(\mathbf{0},\left(\mathcal{J}^{\tau}\left(\vartheta^{\star}\right)\right)^{-1}\left(\mathcal{I}^{\tau}\left(\vartheta^{\star}\right)\right)\left(\mathcal{J}^{\tau}\left(\vartheta^{\star}\right)\right)^{-1}
		\right).
	\end{align*}

The asymptotic normality and convergence of moments of the hybrid multi-step estimators with the initial Bayes type estimators are as follows.

%Using the same notation for $\left(\zeta_{0},\zeta_{1},\zeta_{2}\right)$ as \citet{Nakakita-Uchida-2018c}, 
%we have the next result.
\begin{theorem}\label{result} 
 Assume [A1]-[A6].
Then, under $k_{\tau_{3},n}\Delta_{\tau_{3},n}^{2}\to0$, 
%the following convergence in law holds:
\begin{align*}
    \left(\sqrt{n}\left(\hat{\theta}_{\varepsilon,n}-\theta_{\varepsilon}^{\star}\right),\sqrt{k_{\tau_{3},n}}\left(\hat{\alpha}_{J_{1},n}-\alpha^{\star}\right),\sqrt{T_{n}}\left(\hat{\beta}_{J_{2},n}-\beta^{\star}\right)\right)
    \to^{\mathcal{L}}
    \left(\zeta_{0},\zeta_{1},\zeta_{2}\right).
\end{align*}
%    What is more,
   Moreover, 
   %it holds
    \begin{align*}
        \mathbf{E}\left[f\left(\sqrt{n}\left(\hat{\theta}_{\varepsilon,n}-\theta_{\varepsilon}^{\star}\right),\sqrt{k_{\tau_{3},n}}\left(\hat{\alpha}_{J_{1},n}-\alpha^{\star}\right),\sqrt{T_{n}}\left(\hat{\beta}_{J_{2},n}-\beta^{\star}\right)\right)\right]\to \mathbf{E}\left[f\left(\zeta_{0},\zeta_{1},\zeta_{2}\right)\right]
    \end{align*}
%    where $f$ is a continuous and at most polynomial growth function.
    for all continuous functions f of at most polynomial growth.
\end{theorem}

\section{Example and simulation results}

We consider the three-dimensional diffusion process.
%which is introduced in Kaino et al.\ (2017),
%that is,
\begin{eqnarray*}  %\label{sde-ex}
\mathrm{d}X_t &=& b(X_t, \beta)\mathrm{d}t +a(X_t, \alpha)\mathrm{d}w_t, \quad t \geq 0, \quad
X_0 = \left(1,1,1 \right)^{\star},
\end{eqnarray*}
where
\begin{eqnarray*}
b(X_t, \beta) &=& \left(
\begin{array}{ccc}
1- \beta_1 X_{t,1} - 10 \sin(\beta_2 X_{t,2}^2)	\\
1- \beta_3 X_{t,2} - 10 \sin( \beta_4 X_{t,3}^2)	\\
1 - \beta_5 X_{t,3} - 10 \sin (\beta_6 X_{t,1}^2)    \\
\end{array}
\right),
\\
a(X_t,\alpha) &=&
\left(
\begin{array}{ccc}
\sqrt{\alpha_1(2 + \cos(X_{t,3}^2))} & 0 & 0 \\
0 & \sqrt{\alpha_2(2 + \cos(X_{t,1}^2))} & 0 \\
0 &0 &  \sqrt{\alpha_3(2 + \cos(X_{t,2}^2))} \\
\end{array}
\right).
\end{eqnarray*}
Moreover,
%$\alpha=(\alpha_{1}, \alpha_{2}, \alpha_{3})$
%and
%$\beta=(\beta_1, \beta_2, \beta_3, \beta_4, \beta_5, \beta_6, \beta_7, \beta_8, \beta_9)$ 
%are unknown parameters, and 
the true parameter values are 
%$(\beta_{1,0},\beta_{2,0},\beta_{3,0},\beta_{4,0},\alpha_{1,0},\alpha_{2,0},\alpha_{3,0},\alpha_{4,0}) = (3,6,9,12,15,18,21,24)$.
$$(\beta_1^*, \beta_2^*, \beta_3^*, \beta_4^*, \beta_5^*, \beta_6^*)
= (1,2,2,3,3,4)$$
and 
$(\alpha_1^*, \alpha_2^*, \alpha_3^*)= (1,2,3)$.
The parameter space is $\Theta =[0.01,10]^{9}$.

The noisy data $\left\{Y_{ih_{n}}\right\}_{i=0,\ldots,n}$ are defined as for all $i=0,\ldots,n$,
\begin{align*}
    Y_{ih_{n}} := X_{ih_{n}} + \Lambda^{1/2}\varepsilon_{ih_{n}},
\end{align*}
where $n=5 \times 10^7$, $h_n=\frac{4}{10^6}$,  $T=nh_n= 200$, 
$\Lambda=10^{-3} I_3$, $I_3$ is the $3\times 3$-identity matrix, 
$\left\{\varepsilon_{ih_{n}}\right\}_{i=0,\ldots,n}$ is the i.i.d.\ sequence of $3$-dimensional normal random vectors with $\mathbf{E}\left[\varepsilon_{0}\right]=\mathbf{0}$ and $\mathrm{Var}\left(\varepsilon_{0}\right)=I_{3}$.

For the true model, 100 independent sample paths are generated by the Euler-Maruyama scheme,
and the mean and the standard deviation (s.d.) for the estimators in Theorems 1 and 2 are computed
and shown in Tables 1-9 below.
The personal computer with Intel i7-6950X (3.00GHz) was used
%Intel(R) Core(TM) i7-6950X CPU @ 3.00GHz 3.00GHz 
for the simulations.
%we used the personal computer with Intel i7-5930K (3.5GHz base clock). 
In each table,
the time means the computation time of estimators for one sample path.
%In the same way as the simulations in Section 1, 

%The time  in each table is the computation time of estimator for one sample path.
%The personal computer
%with Intel i7-5930K (3.5GHz base clock) was used for simulations.

Table 1 shows the simulation results of the estimator $\hat{\Lambda}_n =(\Lambda_{n,i,j})_{i,j=1,2,3}$ 
of $\Lambda =(\Lambda_{ij})_{i,j=1,2,3}$.

Tables 2 and 3 show 
the simulation results of the adaptive ML type estimator $(\hat{\alpha}_{A, n}, \hat{\beta}_{A, n})$
with the initial value being the true value, where
\begin{eqnarray*}
\hat{\alpha}_{A, n} &=& \arg \sup_{\alpha \in \Theta_1}   \mathbb{H}_{1, n}\left(\alpha|\hat{\Lambda}_n \right), \\
\hat{\beta}_{A, n} &=& \arg \sup_{\beta \in \Theta_2} \mathbb{H}_{2, n}\left(\beta| \hat{\alpha}_{A, n} \right),
\end{eqnarray*}
the quasi log likelihood functions are  that
{\small\begin{align*}
&\mathbb{H}_{1, n}\left(\alpha|\Lambda\right)
=-\frac{1}{2}\sum_{j=1}^{k_{\tau_{3},n}-2}
\left(\left(\frac{2}{3}\Delta_{\tau_{3},n}A_{n}^{\tau_{3}}\left(\bar{Y}_{\tau_{3},j-1},\alpha,\Lambda\right)\right)^{-1}\left[\left(\bar{Y}_{\tau_{3},j+1}-\bar{Y}_{\tau_{3},j}\right)^{\otimes 2}\right]\right.\\
&\hspace{5cm}\left.+\log\det A_{n}^{\tau_{3}}\left(\bar{Y}_{\tau_{3},j-1},\alpha,\Lambda\right)\right),\\
&\mathbb{H}_{2, n}\left(\beta|\alpha\right)
=-\frac{1}{2}\sum_{j=1}^{k_{\tau_{3},n}-2}
\left(\left(\Delta_{\tau_{3},n}A\left(\bar{Y}_{\tau_{3},j-1},\alpha\right)\right)^{-1}\left[\left(\bar{Y}_{\tau_{3},j+1}-\bar{Y}_{\tau_{3},j}-\Delta_{\tau_{3},n}b\left(\bar{Y}_{\tau_{3},j-1},\beta\right)\right)^{\otimes 2}\right]\right),
\end{align*}} 
\noindent 
the local mean $\left\{\bar{Y}_{\tau,j}\right\}_{j=0,\ldots,k_{\tau,n}-1}$ is defined as
\begin{align*}
	\bar{Y}_{\tau,j}&=\frac{1}{p_{\tau,n}}\sum_{i=0}^{p_{\tau,n}-1}Y_{j\Delta_{\tau,n}+ih_{n}}.
\end{align*}
Here $\tau_3=2.0$, $k_{\tau_3, n}=10^5$, $p_{\tau_3, n}=500$, 
$\Delta_{\tau_3,n}=2\times10^{-3}$, $T=k_{\tau_3, n} \Delta_{\tau_3,n}=200$,
$A_{n}^{\tau_{3}}\left(x,\alpha,\Lambda\right)=
A\left(x,\alpha\right)+3\Delta_{\tau_3, n}^{\frac{2-\tau_{3}}{\tau_{3}-1}}\Lambda
=A(x, \alpha)=a a^T (x,\alpha)$. 
The adaptive ML type estimator $(\hat{\alpha}_{A, n}, \hat{\beta}_{A, n})$
are obtained by means of  {\bf optim()} based on the "L-BFGS-B" method 
%with the initial value being the true value
in the R Language.

%

%\clearpage

\begin{table}[h]
\caption{estimator of $\Lambda$}
\begin{tabular}{c|cccccc||c} \hline
		&$\hat{\Lambda}_{11}(0.001)$&$\hat{\Lambda}_{12}(0)$&{$\hat{\Lambda}_{13}(0)$}& 
		$\hat{\Lambda}_{22}(0.001)$&$\hat{\Lambda}_{23}(0)$&$\hat{\Lambda}_{33}(0.001)$ & time(sec.)
%&time(sec.)
\\ \hline
 &0.001&  0.000&  0.000& 0.001& 0.000& 0.001  &
% &  
 \\
  &(0.000) &(0.000) &(0.000) &(0.000) &(0.000) &(0.000)  & 0.47
 %&70 
 \\   \hline
\end{tabular}
\end{table}

\begin{table}[h]
\begin{center}
\caption{adaptive ML type estimator of $\alpha$ with the initial value being the true value}
\begin{tabular}{c|ccc||c} \hline
		&$\hat{\alpha}_{1}(1)$&$\hat{\alpha}_{2}(2)$&$\hat{\alpha}_{3}(3)$ & time(sec.) \\ \hline
 &1.048 &2.057 &3.053  & \\
 true & (0.005) & (0.01) & (0.014)  & 17 \\   \hline
\end{tabular}
\end{center}
\end{table}

\begin{table}[h]
\caption{adaptive ML type estimator of $\beta$ with the initial value being the true value}
\begin{tabular}{c|cccccc||c} \hline
		&$\hat{\beta}_{1}(1)$&$\hat{\beta}_{2}(2)$&$\hat{\beta}_{3}(2)$&$\hat{\beta}_{4}(3)$&
		$\hat{\beta}_{5}(3)$&$\hat{\beta}_{6}(4)$ &time(sec.)
\\ \hline
 &1.021 &1.988 &2.043 &2.953 &3.048 &3.984&   
% &  
 \\
 true &(0.046) &(0.01)  &(0.099) &(0.038) &(0.162) &(0.028) &  61  
 %&70 
 \\   \hline
\end{tabular}
\end{table}

%\clearpage

From Tables 1-3, we see  that all estimators have good behaviour.

%\clearpage

%\clearpage

\begin{table}[h]
\begin{center}
\caption{adaptive ML type estimator of $\alpha$ with the initial value being the uniform random number on $\Theta$}
\begin{tabular}{c|ccc||c} \hline
		&$\hat{\alpha}_{1}(1)$&$\hat{\alpha}_{2}(2)$&$\hat{\alpha}_{3}(3)$ & time(sec.) \\ \hline
 &1.048 &2.057 &3.053  & \\
 unif &(0.005) &(0.01)  &(0.014)  &  33  \\   \hline
\end{tabular}
\end{center}
%\end{table}

%\begin{table}[h]
\caption{adaptive ML type estimator of $\beta$ with the initial value being the uniform random number on $\Theta$}
\begin{tabular}{c|cccccc||c} \hline
		&$\hat{\beta}_{1}(1)$&$\hat{\beta}_{2}(2)$&$\hat{\beta}_{3}(2)$&$\hat{\beta}_{4}(3)$&
		$\hat{\beta}_{5}(3)$&$\hat{\beta}_{6}(4)$ &time(sec.)
\\ \hline
 &0.384  &2.634 &1.344 &2.860  &2.415 &4.948 &   
% &  
 \\
 true &(0.246) &(3.934) &(0.398) &(3.847) &(0.470)  &(4.496) &  53  
  \\   \hline
\end{tabular}
\end{table}

\noindent
Tables 4 and 5  show the simulation results of the adaptive ML type estimator
$(\hat{\alpha}_{A, n}, \hat{\beta}_{A, n})$ 
with the initial value being the uniform random number on $\Theta$.
Most of estimators of $\beta$ have considerable biases
%which means that  the optimization fails. 
since the initial value may be far from the true value.
As is well known,
it is essential to select an appropriate initial value for {optimisation}.

Tables 6 and 7  show the simulation results of the initial Bayes type estimators with uniform priors 
defined as 
\begin{align*}
%	\hat{\Lambda}_{n}&:=\frac{1}{2n}\sum_{i=0}^{n-1}\left(Y_{\left(i+1\right)h_{n}}-Y_{ih_{n}}\right)^{\otimes2},\\
	\tilde{\alpha}_{q_{1},\tau_{1},n}^{\left(0\right)}&=\frac{\int_{\Theta_{1}}\alpha\exp\left(\mathbb{H}_{1,\tau_{1},n}^{\left(0\right)}\left(\alpha|\hat{\Lambda}_{n}\right)\right)
	%\pi_{1}\left(\alpha\right)
	\mathrm{d}\alpha}{\int_{\Theta_{1}}\exp\left(\mathbb{H}_{1,\tau_{1},n}^{\left(0\right)}\left(\alpha|\hat{\Lambda}_{n}\right)\right)
	%\pi_{1}\left(\alpha\right)
	\mathrm{d}\alpha},\\
	\tilde{\beta}_{q_{2},\tau_{2},n}^{\left(0\right)}&=\frac{\int_{\Theta_{2}}\beta\exp\left(\mathbb{H}_{2,\tau_{2},n}^{\left(0\right)}\left(\beta\right)\right)
	%\pi_{2}\left(\beta\right)
	\mathrm{d}\beta}{\int_{\Theta_{2}}\exp\left(\mathbb{H}_{2,\tau_{2},n}^{\left(0\right)}\left(\beta\right)\right)
	%\pi_{2}\left(\beta\right)
	\mathrm{d}\beta},
\end{align*}
where
\begin{align*}
	\mathbb{H}_{1,\tau_{1},n}^{\left(0\right)}\left(\alpha|\Lambda\right)&=\frac{1}{\underline{k}_{\eta_{1},\tau_{1},n}^{1-2q_{1}}}\mathbb{W}_{1,\tau_{1},n}\left(\alpha|\Lambda\right),&
	\mathbb{H}_{2,\tau_{2},n}^{\left(0\right)}\left(\beta\right)&=\frac{1}{\underline{T}_{\eta_{2},n}^{1-2q_{2}}}\mathbb{W}_{2,\tau_{2},n}\left(\beta\right),
\end{align*}
%Let us set $q_{1},q_{2}\in\left(0,1/2\right]$ and the following quasi likelihood functions:
\begin{align*}
	\mathbb{W}_{1,\tau_{1},n}\left(\alpha|\Lambda\right)&:=-\frac{1}{2}\sum_{j=1}^{\underline{k}_{\eta_{1},\tau_{1},n}-2}\left\|\Delta_{\tau_{1},n}^{-1}\left(\bar{Y}_{\tau_{1},j+1}-\bar{Y}_{\tau_{1},j}\right)^{\otimes2}-\frac{2}{3}A_{\tau_{1},n}\left(\bar{Y}_{\tau_{1},j-1},\alpha,\Lambda\right)\right\|^{2},\\
	\mathbb{W}_{2,\tau_{2},n}\left(\beta\right)&:=-\frac{1}{2}\sum_{j=1}^{\underline{k}_{\eta_{2},\tau_{2},n}-2}\Delta_{\tau_{2},n}^{-1}\left|\bar{Y}_{\tau_{2},j+1}-\bar{Y}_{\tau_{2},j}-\Delta_{\tau_{2},n}b\left(\bar{Y}_{\tau_{2},j-1},\beta\right)\right|^{2}.
\end{align*}
Here we set $\eta_1=\eta_2= 61/70$ and the initial Bayes estimator of $\alpha$ are obtained by the reduced data with
$q_1=1/2$, $\tau_1=2.0$, $\underline{k}_{\eta_1, \tau_1, n}=10^4$, $\Delta_{\tau_1, n}=2\times10^{-3}$.
%$T_\alpha=20$, the numbers of MCMC is $10^3$.
Moreover, the initial Bayes estimator of $\beta$ are derived from the reduced data with 
$q_2=1/2$, $\tau_2=2.0$, $\underline{k}_{\eta_2, \tau_2, n}=10^4$, $\Delta_{\tau_2, n}=2\times10^{-3}$,
$\underline{T}_{\eta_2, n} =\underline{k}_{\eta_2, \tau_2, n} \Delta_{\tau_2, n}=20$.

Furthermore, the initial Bayes type estimators of $\alpha$ and $\beta$ are calculated with MpCN method 
proposed by  \citet{Kamatani-2018} for 
$10^3$ and $10^6$ Markov chains and $10^2$  and $10^5$ burn-in iterations, respectively.

\begin{en-text}
Moreover, the hybrid estimator $(\hat{\alpha}_{4,n}^{(3)}, \hat{\beta}_{4,n}^{(4)})$
is given by
\beas
\hat{\alpha}_{4,n}^{(3)} 
&=& \arg \sup_{\alpha \in \Theta_\alpha}  V_{n}^{(3)}(\alpha \ | \  
\hat{\alpha}_{4,n}^{(1)}, \hat{\beta}_{4,n}^{(2)}),
\\
\hat{\beta}_{4,n}^{(4)}  
&=& \arg \sup_{\beta \in \Theta_\beta}  V_{n}^{(4)}(\beta \ | \  
\hat{\alpha}_{4,n}^{(3)}, \hat{\beta}_{4,n}^{(2)}).
\eeas
\end{en-text}

%\clearpage

\begin{table}[h]
\begin{center}
\caption{initial Bayes type estimator of $\alpha$ with reduced data ($\underline{k}_{\eta_1, \tau_1, n}=10^4$)}
\begin{tabular}{c|ccc||c} \hline
		&$\tilde{\alpha}^{(0)}_{1}(1)$&$\tilde{\alpha}^{(0)}_{2}(2)$&$\tilde{\alpha}^{(0)}_{3}(3)$ & time(sec.) \\ \hline
 &1.029 &2.015 &3.019  & \\
  &(0.023) &(0.035) &(0.047) & 22  \\   \hline
\end{tabular}
\end{center}
\end{table}

\begin{table}[h]
\caption{initial Bayes type estimator of $\beta$ with reduced data 
($\underline{k}_{\eta_2, \tau_2, n}=10^4$,  $\underline{T}_{\eta_2, n} =20$)}
\begin{tabular}{c|cccccc||c} \hline
		&$\hat{\beta}^{(0)}_{1}(1)$&$\hat{\beta}^{(0)}_{2}(2)$&$\hat{\beta}^{(0)}_{3}(2)$&$\hat{\beta}^{(0)}_{4}(3)$&
		$\hat{\beta}^{(0)}_{5}(3)$&$\hat{\beta}^{(0)}_{6}(4)$ &time(min.)
\\ \hline
 & 1.077 &1.990  &2.169 &2.933 &3.111 &3.989 &  \\
 &(0.197) &(0.034) &(0.453) &(0.130)  &(0.523) &(0.108)  &46  \\   \hline
\end{tabular}
\end{table}

%\noindent
%Tables 6 and 7 show the simulation results of the initial Bayes estimators 
%with reduced data ($\underline{k}_{\eta_1, \tau_1, n}=10^4$ and $\underline{k}_{\eta_2, \tau_2, n}=10^4$).

Since we set that $\eta_1=\eta_2= 61/70$, {$\gamma=\gamma'=7/10$}, $\tau_{1}=\tau_{3}=2.0$, $q_{1}=q_2=1/2$, 
one has that 
\begin{eqnarray*}
J_{1} &=& \lfloor -\log_{2}\left( q_{1}\left(\eta_{1}-\gamma/\tau_{1}\right)/\left(1-\gamma'/\tau_{3}\right)\right)\rfloor= 1,
\\
J_{2} &=& \lfloor -\log_{2} \left( q_{2}\left(\eta_{2}-\gamma\right)/\left(1-\gamma' \right) \right)\rfloor=1.
\end{eqnarray*}
%Moreover, 
%{\color{black} Set $(\hat{\alpha}_{4,n}^{(1)}, \hat{\beta}_{4,n}^{(2)}) = 
%(\check{\alpha}_{4,n}^{(k_1)}, \check{\beta}_{4,n}^{(k_2)})$, where
Tables 8 and 9  show the simulation results of the hybrid multi-step estimators
$(\hat{\alpha}_{J_1,n}, \hat{\beta}_{J_2,n})$ with
the initial estimator $(\tilde{\alpha}_{q_{1},\tau_{1},n}^{\left(0\right)}, \tilde{\beta}_{q_{2},\tau_{2},n}^{\left(0\right)})$
in Tables 6 and 7.

\begin{table}[h]
\begin{center}
\caption{hybrid multi-step estimator of $\alpha$ with the initial Bayes type estimator 
 ($\underline{k}_{\eta_2, \tau_2, n}=10^4$)}
\begin{tabular}{c|ccc||c} \hline
		&$\hat{\alpha}_{1}(1)$&$\hat{\alpha}_{2}(2)$&$\hat{\alpha}_{3}(3)$ & time(sec.) \\ \hline
 &1.048 &2.057 &3.053   & \\
 &(0.005) &(0.010) &(0.014)  &13  \\   \hline
\end{tabular}
\end{center}
\end{table}

\begin{table}[h]
\caption{hybrid multi-step estimator of $\beta$ with the initial Bayes type estimator
 ($\underline{k}_{\eta_2, \tau_2, n}=10^4$,  $\underline{T}_{\eta_2, n} =20$))}
\begin{tabular}{c|cccccc||c} \hline
		&$\hat{\beta}_{1}(1)$&$\hat{\beta}_{2}(2)$&$\hat{\beta}_{3}(2)$&$\hat{\beta}_{4}(3)$&
		$\hat{\beta}_{5}(3)$&$\hat{\beta}_{6}(4)$ &time(sec.)
\\ \hline
 &1.021 &1.988 &2.044 &2.953 &3.048 &3.983 &   \\
&(0.046) &(0.010) &(0.099) &(0.038) &(0.162) &(0.028)& 70 \\   \hline
\end{tabular}
\end{table}

%\clearpage

\noindent
%Tables 8 and 9 show the simulation results of the multi-step estimators 
%based on the initial Bayes estimators in Tables 6 and 7.
From Tables 8 and 9, we can see that the hybrid multi-step estimators with the initial Bayes estimators 
improve the initial Bayes estimators in Tables 6 and 7.
It is worth mentioning that the performance  of the hybrid multi-step estimator 
with the initial Bayes estimator is almost the same as that of the estimator in Tables 2 and 3.
%In this example, we see from the simulation results that 
%the type 4 Bayes  estimator of $\beta$ with the thinned data of the sample size $N_1 =5000$
%has a bias, and
%the type 4 Bayes estimator of $\alpha$ with the reduced data of the sample size $N_1 =10000$, 
%and the hybrid estimator with the type 4 initial Bayes estimators
%have good performance.

\section{Concluding remarks}
In this paper, we have provided the hybrid estimation for
noisily ergodic diffusion processes based on ultra high frequency data
from the viewpoint of computational cost.
In order to get the adaptive ML type estimators,
%we need some {optimisation} of the quasi likelihood function and 
%it is important to select a suitable initial value.
%It is difficult to find a suitable initial value
%when the dimension of the parameter space is large.
we need optimisation of the quasi likelihood function and
selecting a suitable initial value is important,
but it comes difficult when the dimension of the parameter space is large. 
While the computation of the Bayes type estimator is generally
free from  a choice of the initial value, 
there is a serious problem that it takes so much time to compute
the Bayes type estimator when the sample size is large. 
 \citet{Kamatani-Uchida-2015} 
proposed the multi-step ML type estimator based on the initial Bayes type estimator with the full data, 
and  \citet{Kaino-et-al-2017} and  \citet{Kaino-Uchida-2018a, Kaino-Uchida-2018b} studied 
the adaptive ML type estimator with the initial Bayes type estimator
derived from the reduced data by applying the result of 
 \citet{Kutoyants-2017} to high frequency data analysis.
%As seen in Section 1, the hybrid estimator based on the initial Bayes type estimator 
%with the reduced data of the sample size $n_1=20000$
%works very well for an ergodic diffusion process from high frequency data with $h_n=1/390$. 
%However, in  the case of ergodic diffusion process from ultra high frequency data with $h_n=1/3900$, 
%the initial Bayes estimator of the drift parameter based on 
%the reduced data of the sample size $n_2=10000$
%obtained from ultra high frequency data
%does not work since $n_2 h_n \approx  2.59$ is not large enough for ergodicity.
In this paper, we have proposed the initial Bayes type estimator with the reduced data 
based on the local means obtained from the high frequency data with noise 
and hybrid multi-step estimator with the initial Bayes type estimator.
The proposed hybrid multi-step estimators have asymptotic normality and convergence of moments
and we see from the numerical examples in Section 4 that 
they have good performance.

\section{Proof}
\subsection{Some lemmas with respect to local means} We note some useful lemmas for moment evaluations.
Most of them are proved in \citet{Nakakita-Uchida-2018b} and \citet{Nakakita-Uchida-2018c}.
We denote $\tau=\tau_{i}$ for $i=1,2,3$, and define the following random variables:
\begin{align*}
\zeta_{\tau,j+1,n}:=
\frac{1}{p_{\tau,n}}\sum_{i=0}^{p_{\tau,n}-1}\int_{j\Delta_{\tau,n}+ih_{n}}^{\left(j+1\right)\Delta_{\tau,n}}\mathrm{d}w_{s},\quad
\zeta_{\tau,j+2,n}':=
\frac{1}{p_{\tau,n}}\sum_{i=0}^{p_{\tau,n}-1}\int_{\left(j+1\right)\Delta_{\tau,n}}^{\left(j+1\right)\Delta_{\tau,n}+ih_{n}}\mathrm{d}w_{s}.
\end{align*}

The next lemma is Lemma 11 in \cite{Nakakita-Uchida-2018b}.

\begin{lemma}\label{EvalZeta}
	$\zeta_{\tau,j+1,n}$ and $\zeta_{\tau,j+1,n}'$ are $\mathcal{G}_{j+1,n}^{\tau}$-measurable, independent of $\mathcal{G}_{j,n}^{\tau}$ and Gaussian.These variables have the next decomposition:
	\begin{align*}
	\zeta_{\tau,j+1,n}&=\frac{1}{p_{\tau,n}}\sum_{k=0}^{p_{\tau,n}-1}\left(k+1\right)\int_{j\Delta_{\tau,n}+kh_{n}}^{j\Delta_{\tau,n}+\left(k+1\right)h_{n}}\mathrm{d}w_{s},\\
	\zeta_{\tau,j+1,n}'&=\frac{1}{p_{\tau,n}}\sum_{k=0}^{p_{\tau,n}-1}\left(p_{\tau,n}-k-1\right)
	\int_{j\Delta_{\tau,n}+kh_{n}}^{j\Delta_{\tau,n}+\left(k+1\right)h_{n}}\mathrm{d}w_{s}.
	\end{align*}
	The evaluation of the following conditional expectations holds:
	\begin{align*}
	\mathbf{E}\left[\zeta_{\tau,j,n}|\mathcal{G}_{j,n}^{\tau}\right]=\mathbf{E}\left[\zeta_{\tau,j+1,n}'|\mathcal{G}_{j,n}^{\tau}\right]&=\mathbf{0},\\
	\mathbf{E}\left[\zeta_{\tau,j+1,n}\left(\zeta_{\tau,j+1,n}\right)^{T}|\mathcal{G}_{j,n}^{\tau}\right]&=m_{\tau,n}\Delta_{\tau,n}I_{r}, \\
	\mathbf{E}\left[\zeta_{\tau,j+1,n}'\left(\zeta_{\tau,j+1,n}'\right)^{T}|\mathcal{G}_{j,n}^{\tau}\right]&=m_{\tau,n}'\Delta_{\tau,n}I_{r}, \\
	\mathbf{E}\left[\zeta_{\tau,j+1,n}\left(\zeta_{\tau,j+1,n}'\right)^{T}|\mathcal{G}_{j,n}^{\tau}\right]&=\chi_{\tau,n}\Delta_{\tau,n}I_{r}, 
    \end{align*}
	where $m_{\tau,n}=\left(\frac{1}{3}+\frac{1}{2p_{\tau,n}}+\frac{1}{6p_{\tau,n}^2}\right)$, $m_{\tau,n}'=\left(\frac{1}{3}-\frac{1}{2p_{\tau,n}}+\frac{1}{6p_{\tau,n}^2}\right)$, and $\chi_{\tau,n}=\frac{1}{6}\left(1-\frac{1}{p_{\tau,n}^2}\right)$.
\end{lemma}

The next lemma is from \citet{Nakakita-Uchida-2018c}.

\begin{lemma}\label{CED}
Assume [A1]-[A2] and [A5]-[A6].
	Moreover, assume the components of the functions $f,g\in C^{2}\left(\mathbf{R}^{d};\ \mathbf{R}\right)$, $\partial_{x}f$, $\partial_{x}g$, $\partial_{x}^{2}f$ $\partial_{x}^{2}g$ are polynomial growth functions. Then we have
	\begin{align*}
	\left|\mathbf{E}\left[f\left(\bar{Y}_{\tau,j}\right)g\left(X_{\left(j+1\right)\Delta_{\tau,n}}\right)-f\left(X_{j\Delta_{\tau,n}}\right)g\left(X_{j\Delta_{\tau,n}}\right)|\mathcal{H}_{j,n}^{\tau}\right]\right|\le C\Delta_{\tau,n}\left(1+\left|X_{j\Delta_{\tau,n}}\right|\right)^{C}.
	\end{align*}
\end{lemma}

The next lemma is from \citet{Nakakita-Uchida-2018b} and \citet{Nakakita-Uchida-2018c}.
\begin{lemma}\label{localMeanExpansion}
Assume [A1]-[A2] and [A5]-[A6].
\begin{enumerate}
\item The next expansion holds:
\begin{align*}
	\bar{Y}_{\tau,j+1}-\bar{Y}_{\tau,j}&=\Delta_{\tau,n}b\left(X_{j\Delta_{\tau,n}}\right)+a\left(X_{j\Delta_{\tau,n}}\right)\left(\zeta_{\tau,j+1,n}+\zeta_{\tau,j+2,n}'\right)\\
	&\qquad+e_{\tau,j,n}+\left(\Lambda_{\star}\right)^{1/2}\left(\bar{\varepsilon}_{\tau,j+1}-\bar{\varepsilon}_{\tau,j}\right),
\end{align*}
where $e_{\tau,j,n}$ is a $\mathcal{H}_{j+2,n}^{\tau}$-measurable random variable such that $\left|\mathbf{E}\left[e_{\tau,j,n}|\mathcal{H}_{j,n}^{\tau}\right]\right|\le C\Delta_{\tau,n}^{2}\left(1+\left|X_{j\Delta_{\tau,n}}\right|^{C}\right)$ and $\left\|e_{\tau,j,n}\right\|_{p}\le C\left(p\right)\Delta_{\tau,n}$ for $j=1,\ldots,k_{n}-2$, $n\in\mathbf{N}$ and $p\ge 1$.
\item For any $p\ge 1$ and $\mathcal{H}_{j,n}^{\tau}$-measurable $\mathbf{R}^{d}\otimes\mathbf{R}^{r}$-valued random variable $\mathbb{B}_{j,n}\in \bigcap_{p>0}L^{p}\left(P_{\theta^{\star}}\right)$, we have the next $L^{p}$-boundedness
\begin{align*}
\mathbf{E}\left[\left|\sum_{j=1}^{k_{\tau,n}-2}\mathbb{B}_{j,n}\left[e_{\tau,j,n}\left(\zeta_{\tau,j+1,n}+\zeta_{\tau,j+2,n}'\right)^{T}\right]\right|^{p}\right]^{1/p}\le C\left(p\right)k_{n}\Delta_{\tau,n}^{2}.
\end{align*}
\item For any $p\ge 1$ and $\mathcal{H}_{j,n}^{\tau}$-measurable $\mathbf{R}^{d}\otimes\mathbf{R}^{d}$-valued random variable $\mathbb{C}_{j,n}\in \bigcap_{p>0}L^{p}\left(P_{\theta^{\star}}\right)$, we have the next $L^{p}$-boundedness
\begin{align*}
\mathbf{E}\left[\left|\sum_{j=1}^{k_{\tau,n}-2}\mathbb{C}_{j,n}\left[e_{\tau,j,n}\right]\right|^{p}\right]^{1/p}\le C\left(p\right)k_{n}\Delta_{\tau,n}^{3/2}.
\end{align*}
\end{enumerate}
\end{lemma}

We define for any $\tau\in\left(1,2\right]$ and $j=0,\ldots,k_{\tau,n}-2$,
\begin{align*}
	&\widehat{A}_{\tau,j,n}\\
	&:=\Delta_{\tau,n}^{-1}\left[\left(m_{\tau,n}+m_{\tau,n}'\right)^{-\frac{1}{2}}a\left(X_{j\Delta_{\tau,n}}\right)
	\left(\zeta_{\tau,j+1,n}+\zeta_{\tau,j+2,n}'\right)+\sqrt{\frac{3}{2}}\left(\Lambda_{\star}\right)^{\frac{1}{2}}\left(\bar{\varepsilon}_{\tau,j+1}-\bar{\varepsilon}_{\tau,j}\right)\right]^{\otimes2}.
\end{align*}

\begin{lemma}\label{approxQuadratic}
Assume [A1]-[A2] and [A5]-[A6].
	Moreover, assume $M:\mathbf{R}^{d}\times \Xi \to \mathbf{R}^{d}\otimes\mathbf{R}^{d}$ is a polynomial growth function uniformly in $\vartheta$. Then, for $\underline{k}\le k_{\tau,n}$,
	\begin{align*}
		\left\|\sum_{j=1}^{\underline{k}-2}
		M\left(
				\bar{Y}_{\tau,j-1},
				\alpha
		\right)
		\left[
			\Delta_{\tau,n}^{-1}\left(\bar{Y}_{\tau,j+1}-\bar{Y}_{\tau,j}\right)^{\otimes2}
			-\frac{2}{3}\widehat{A}_{\tau,j,n}
		\right]\right\|_{p}\le C\left(p\right)\underline{k}\Delta_{\tau,n}.
	\end{align*}
\end{lemma}

\begin{proof} We have
\begin{align*}
&\mathbf{E}\left[\left\|\Delta_{\tau,n}\widehat{A}_{\tau,j,n}-\frac{3}{2}\left[a\left(X_{j\Delta_{\tau,n}}\right)\left(\zeta_{\tau,j+1,n}+\zeta_{\tau,j+2,n}\right)+\left(\Lambda_{\star}\right)^{\frac{1}{2}}\left(\bar{\varepsilon}_{\tau,j+1}-\bar{\varepsilon}_{\tau,j}\right)\right]^{\otimes2}\right\|^{p}\right]^{1/p}\\
&\le \left|\frac{1}{m_{\tau,n}+m_{\tau,n}'}-\frac{3}{2}\right|
\mathbf{E}\left[\left\|
\left[a\left(X_{j\Delta_{\tau,n}}\right)\left(\zeta_{\tau,j+1,n}+\zeta_{\tau,j+2,n}'\right)\right]^{\otimes2}
\right\|^{p}\right]^{1/p}\\
&\quad+\sqrt{6}\left|\sqrt{\frac{1}{m_{\tau,n}+m_{\tau,n}'}}-\sqrt{\frac{3}{2}}\right|\\
&\qquad\times\mathbf{E}\left[\left\|
\left[
a\left(X_{j\Delta_{\tau,n}}\right)\left(\zeta_{\tau,j+1,n}+\zeta_{\tau,j+2,n}'\right)\left(\bar{\varepsilon}_{\tau,j+1}-\bar{\varepsilon}_{\tau,j}\right)^{T}\left(\Lambda_{\star}\right)^{1/2}
\right]
\right\|^{p}
\right]^{1/p}\\
&\le \left|\frac{1}{2/3+1/\left(3p_{\tau,n}^{2}\right)}-\frac{1}{2/3}\right|C\left(p\right)\Delta_{\tau,n}+\left|\sqrt{\frac{1}{2/3+1/\left(3p_{\tau,n}^{2}\right)}}-\sqrt{\frac{1}{2/3}}\right|\frac{C\left(p\right)\Delta_{\tau,n}^{1/2}}{p_{\tau,n}^{1/2}}\\
&\le \frac{C\left(p\right)\Delta_{\tau,n}}{p_{\tau,n}^{2}}
\end{align*}
with Taylor's expansion for $f_{1}(x)=1/x$ and $f_{2}(x)=\sqrt{1/x}$ around $x=2/3$. Using this evaluation, we obtain
\begin{align*}
	&\mathbf{E}\left[\left|\sum_{j=1}^{\underline{k}-2}
	M\left(
	\bar{Y}_{\tau,j-1},
	\alpha
	\right)
	\left[
	\Delta_{\tau,n}^{-1}\left(\bar{Y}_{\tau,j+1}-\bar{Y}_{\tau,j}\right)^{\otimes2}
	-\frac{2}{3}\widehat{A}_{\tau,j,n}
	\right]\right|^{p}\right]^{1/p}\\
	&\le \sum_{j=1}^{\underline{k}-2}\left\|M\left(
	\bar{Y}_{\tau,j-1},
	\alpha
	\right)\right\|_{2p}\\
	&\hspace{1.5cm}\times\left\|
	\widehat{A}_{\tau,j,n}-\frac{3}{2\Delta_{\tau,n}}\left[a\left(X_{j\Delta_{\tau,n}}\right)\left(\zeta_{\tau,j+1,n}+\zeta_{\tau,j+2,n}\right)+\left(\Lambda_{\star}\right)^{\frac{1}{2}}\left(\bar{\varepsilon}_{\tau,j+1}-\bar{\varepsilon}_{\tau,j}\right)\right]^{\otimes2}\right\|_{2p}\\
	&\quad+\mathbf{E}\left[\left|\sum_{j=1}^{\underline{k}-2}
	M\left(
	\bar{Y}_{\tau,j-1},
	\alpha
	\right)
	\left[
	\Delta_{\tau,n}^{-1}\left(\Delta_{\tau,n}b\left(X_{j\Delta_{\tau,n}}\right)+e_{\tau,j,n}\right)^{\otimes2}
	\right]\right|^{p}\right]^{1/p}\\
	&\quad+\mathbf{E}\left[\left|\sum_{j=1}^{\underline{k}-2}
	M\left(
	\bar{Y}_{\tau,j-1},
	\alpha
	\right)
	\left[
	\Delta_{\tau,n}^{-1}\left(\Delta_{\tau,n}b\left(X_{j\Delta_{\tau,n}}\right)\right)\left(\zeta_{\tau,j+1,n}+\zeta_{\tau,j+2,n}'\right)^{T}a\left(X_{j\Delta_{\tau,n}}\right)^{T}
	\right]\right|^{p}\right]^{1/p}\\
	&\quad+\mathbf{E}\left[\left|\sum_{j=1}^{\underline{k}-2}
	M\left(
	\bar{Y}_{\tau,j-1},
	\alpha
	\right)
	\left[
	\Delta_{\tau,n}^{-1}e_{\tau,j,n}\left(\zeta_{\tau,j+1,n}+\zeta_{\tau,j+2,n}'\right)^{T}a\left(X_{j\Delta_{\tau,n}}\right)^{T}
	\right]\right|^{p}\right]^{1/p}\\
	&\quad+\mathbf{E}\left[\left|\sum_{j=1}^{\underline{k}-2}
	M\left(
	\bar{Y}_{\tau,j-1},
	\alpha
	\right)
	\left[
	\Delta_{\tau,n}^{-1}\left(\Delta_{\tau,n}b\left(X_{j\Delta_{\tau,n}}\right)+e_{\tau,j,n}\right)\left(\bar{\varepsilon}_{\tau,j+1}-\bar{\varepsilon}_{\tau,j}\right)^{T}\left(\Lambda_{\star}\right)^{1/2}
	\right]\right|^{p}\right]^{1/p}.
\end{align*}
Because of the evaluation above, it holds
\begin{align*}
	&\sum_{j=1}^{\underline{k}-2}\left\|M\left(
	\bar{Y}_{\tau,j-1},
	\alpha
	\right)\right\|_{2p}\\
	&\hspace{1cm}\times\left\|
	\widehat{A}_{\tau,j,n}-\frac{3}{2\Delta_{\tau,n}}\left[a\left(X_{j\Delta_{\tau,n}}\right)\left(\zeta_{\tau,j+1,n}+\zeta_{\tau,j+2,n}\right)+\left(\Lambda_{\star}\right)^{\frac{1}{2}}\left(\bar{\varepsilon}_{\tau,j+1}-\bar{\varepsilon}_{\tau,j}\right)\right]^{\otimes2}\right\|_{2p}\\
	&\le \frac{C\left(p\right)\underline{k}}{p_{\tau,n}^{2}}
\end{align*}
and note that $p_{\tau,n}^{-1}\le \Delta_{\tau,n}$. With triangle inequality and H\"{o}lder's one,
\begin{align*}
	&\mathbf{E}\left[\left|\sum_{j=1}^{\underline{k}-2}
	M\left(
	\bar{Y}_{\tau,j-1},
	\alpha
	\right)
	\left[
	\Delta_{\tau,n}^{-1}\left(\Delta_{\tau,n}b\left(X_{j\Delta_{\tau,n}}\right)+e_{\tau,j,n}\right)^{\otimes2}
	\right]\right|^{p}\right]^{1/p}\\
	&\le \Delta_{\tau,n}^{-1}\sum_{j=1}^{\underline{k}-2}\left|
	M\left(
	\bar{Y}_{\tau,j-1},
	\alpha
	\right)\right\|_{2p}
	\left\|\Delta_{\tau,n}b\left(X_{j\Delta_{\tau,n}}\right)+e_{\tau,j,n}\right\|_{4p}^{2}\\
	&\le C\left(p\right)\underline{k}\Delta_{\tau,n}.
\end{align*}
In the next place, we can evaluate the following $L^{p}$-norm by the three norms, such that
\begin{align*}
	&\mathbf{E}\left[\left|\sum_{j=1}^{\underline{k}-2}
	M\left(
	\bar{Y}_{\tau,j-1},
	\alpha
	\right)
	\left[
	\Delta_{\tau,n}^{-1}e_{\tau,j,n}\left(\zeta_{\tau,j+1,n}+\zeta_{\tau,j+2,n}'\right)^{T}a\left(X_{j\Delta_{\tau,n}}\right)^{T}
	\right]\right|^{p}\right]^{1/p}\\
	&\le \sum_{i=0}^{2}\mathbf{E}\left[\left|\sum_{1\le 3j+i\le \underline{k}-2}
	M\left(
	\bar{Y}_{\tau,3j+i-1},
	\alpha
	\right)
	\left[
	\Delta_{\tau,n}^{-1}\left(\Delta_{\tau,n}b\left(X_{\left(3j+1\right)\Delta_{\tau,n}}\right)\right)\right.\right.\right.\\
	&\hspace{7cm}\left.\left.\left.\left(\zeta_{\tau,3j+i+1,n}+\zeta_{\tau,3j+i+2,n}'\right)^{T}a\left(X_{\left(3j+1\right)\Delta_{\tau,n}}\right)^{T}
	\right]\right|^{p}\right]^{1/p};
\end{align*}
and then Burkholder's inequality leads to
\begin{align*}
	&\mathbf{E}\left[\left|\sum_{1\le 3j\le \underline{k}-2}
	M\left(
	\bar{Y}_{\tau,3j-1},
	\alpha
	\right)
	\left[
	b\left(X_{3j\Delta_{\tau,n}}\right)\left(\zeta_{\tau,3j+1,n}+\zeta_{\tau,3j+2,n}'\right)^{T}a\left(X_{3j\Delta_{\tau,n}}\right)^{T}
	\right]\right|^{p}\right]^{1/p}\\
	&\le \mathbf{E}\left[\left|\sum_{1\le 3j\le \underline{k}-2}
	\left|M\left(
	\bar{Y}_{\tau,3j-1},
	\alpha
	\right)
	\left[
	b\left(X_{3j\Delta_{\tau,n}}\right)\left(\zeta_{\tau,3j+1,n}+\zeta_{\tau,3j+2,n}'\right)^{T}a\left(X_{3j\Delta_{\tau,n}}\right)^{T}
	\right]\right|^{2}\right|^{p/2}\right]^{1/p}\\
	&\le \left(\sum_{1\le 3j\le \underline{k}-2}
	\left\|\left|M\left(
	\bar{Y}_{\tau,3j-1},
	\alpha
	\right)
	\left[
	b\left(X_{3j\Delta_{\tau,n}}\right)\left(\zeta_{\tau,3j+1,n}+\zeta_{\tau,3j+2,n}'\right)^{T}a\left(X_{3j\Delta_{\tau,n}}\right)^{T}
	\right]\right|^{2}\right\|_{p/2}\right)^{1/2}\\
	&\le \left(\sum_{1\le 3j\le \underline{k}-2}C\left(p\right)\left\|\zeta_{\tau,3j+1,n}+\zeta_{\tau,3j+2,n}'\right\|_{p}^{2}\right)^{1/2}\\
	&\le C\left(p\right)\left(\underline{k}\Delta_{\tau,n}\right)^{1/2}.
\end{align*}
Hence we obtain
\begin{align*}
	&\mathbf{E}\left[\left|\sum_{j=1}^{\underline{k}-2}
	M\left(
	\bar{Y}_{\tau,j-1},
	\alpha
	\right)
	\left[
	\Delta_{\tau,n}^{-1}e_{\tau,j,n}\left(\zeta_{\tau,j+1,n}+\zeta_{\tau,j+2,n}'\right)^{T}a\left(X_{j\Delta_{\tau,n}}\right)^{T}
	\right]\right|^{p}\right]^{1/p}\\
	&\le C\left(p\right)\left(\underline{k}\Delta_{\tau,n}\right)^{1/2}
\end{align*}
and similarly
\begin{align*}
	&\mathbf{E}\left[\left|\sum_{j=1}^{\underline{k}-2}
	M\left(
	\bar{Y}_{\tau,j-1},
	\alpha
	\right)
	\left[
	\Delta_{\tau,n}^{-1}\left(\Delta_{\tau,n}b\left(X_{j\Delta_{\tau,n}}\right)+e_{\tau,j,n}\right)\left(\bar{\varepsilon}_{\tau,j+1}-\bar{\varepsilon}_{\tau,j}\right)^{T}\left(\Lambda_{\star}\right)^{1/2}
	\right]\right|^{p}\right]^{1/p}\\
	&\le C\left(p\right)\left(\frac{\underline{k}}{p_{\tau,n}}\right)^{1/2}.
\end{align*}
Finally because of Lemma \ref{localMeanExpansion}
\begin{align*}
	\mathbf{E}\left[\left|\sum_{j=1}^{\underline{k}-2}
	M\left(
	\bar{Y}_{\tau,j-1},
	\alpha
	\right)
	\left[
	\Delta_{\tau,n}^{-1}e_{\tau,j,n}\left(\zeta_{\tau,j+1,n}+\zeta_{\tau,j+2,n}'\right)^{T}a\left(X_{j\Delta_{\tau,n}}\right)^{T}
	\right]\right|^{p}\right]^{1/p}\le C\left(p\right)\underline{k}\Delta_{\tau,n},
\end{align*}
which completes the proof.
\end{proof}

\subsection{Derivation and evaluation for locally asymptotic quadratic form}

We give the locally asymptotic quadratic form at $\vartheta^{\star}\in\Xi$ for $u_{1}\in\mathbf{R}^{m_{1}}$ and $u_{2}\in\mathbf{R}^{m_{2}}$,
\begin{align*}
	\mathbb{Z}_{1,\tau_{1},n}^{\left(0\right)}\left(u_{1};\vartheta^{\star}\right)
	&:=\exp\left(S_{1,\tau_{1},n}\left(\vartheta^{\star}\right)\left[u_{1}\right]-\frac{1}{2}\Gamma_{1,\tau_{1}}\left(\vartheta^{\star}\right)\left[u_{1}^{\otimes2}\right]+
	r_{1,\tau_{1},n}\left(u_{1};\vartheta^{\star}\right)\right),\\
	\mathbb{Z}_{2,\tau_{2},n}^{\left(0\right)}\left(u_{2};\vartheta^{\star}\right)&:=\exp\left(S_{2,\tau_{2},n}\left(\vartheta^{\star}\right)\left[u_{2}\right]-\frac{1}{2}\Gamma_{2,\tau_{2}}\left(\vartheta^{\star}\right)\left[u_{2}^{\otimes2}\right]+
	r_{2,\tau_{2},n}\left(u_{2};\vartheta^{\star}\right)\right),
\end{align*}
where the residual terms are defined as
\begin{align*}
	r_{1,\tau_{1},n}\left(u_{1};\vartheta^{\ast}\right)&:=\int_{0}^{1}\left(1-s\right)\left\{
	\Gamma_{1,\tau_{1}}\left(\vartheta^{\star}\right)\left[u_{1}^{\otimes2}\right]-\Gamma_{1,\tau_{1},n}\left(\alpha^{\star}+s\underline{k}_{\eta_{1},\tau_{1},n}^{-q_{1}}u_{1};\vartheta^{\star}\right)\left[u_{1}^{\otimes 2}\right]\right\}\mathrm{d}s,\\
	r_{2,\tau_{2},n}\left(u_{2};\vartheta^{\ast}\right)&:=\int_{0}^{1}\left(1-s\right)\left\{
	\Gamma_{2}\left(\vartheta^{\star}\right)\left[u_{2}^{\otimes2}\right]-\Gamma_{2,\tau_{2},n}\left(\beta^{\star}+s\underline{T}_{\eta_{2},n}^{-q_{2}}u_{2};\vartheta^{\star}\right)\left[u_{2}^{\otimes 2}\right]\right\}\mathrm{d}s.
\end{align*}

\begin{proof}[Proof of Lemma \ref{initAlphaLemma}]
We start with the proof of (1). Let us define
\begin{align*}
	\widetilde{S}_{1,\tau_{1},n}\left(\vartheta^{\star}\right)\left[u_{1}\right]
	&:=-\frac{2}{3\underline{k}_{\eta_{1},\tau_{1},n}^{1-q_{1}}}\sum_{j=1}^{\underline{k}_{\eta_{1},\tau_{1},n}-2}\left(\partial_{\alpha}A\left(\bar{Y}_{\tau_{1},j-1},\alpha^{\star}\right)\right)\\
	&\hspace{3cm}\left[u_{1},\Delta_{\tau_{1},n}^{-1}\left(\bar{Y}_{\tau_{1},j+1}-\bar{Y}_{\tau_{1},j}\right)^{\otimes2}
	-\frac{2}{3}A_{\tau_{1},n}\left(\bar{Y}_{\tau_{1},j-1},\alpha^{\star},\Lambda_{\star}\right)\right].
\end{align*}
Because of $\mathbf{E}\left[\left\|\sqrt{n}\left(\hat{\Lambda}_{n}-\Lambda_{\star}\right)\right\|^{p}\right]<\infty$ shown in \citet{Nakakita-Uchida-2018c}, we have the evaluation such that
\begin{align*}
	\sup_{n\in\mathbf{N}}\mathbf{E}\left[\left|S_{1,\tau_{1},n}\left(\vartheta^{\star}\right)-\widetilde{S}_{1,\tau_{1},n}\left(\vartheta^{\star}\right)\right|^{p}\right]
	\le \sup_{n\in\mathbf{N}}C\left(p\right)\frac{\underline{k}_{\eta_{1},\tau_{1},n}^{pq_{1}}}{n^{\frac{p}{2}}} <\infty.
\end{align*}
Hence it is enough to prove that 
%if it holds 
\begin{align*}
	\sup_{n\in\mathbf{N}}\mathbf{E}\left[\left|\widetilde{S}_{1,\tau_{1},n}\left(\vartheta^{\star}\right)\right|^{p}\right]<\infty.
\end{align*}
By Lemma \ref{localMeanExpansion}, we obtain the decomposition
\begin{align*}
	&\widetilde{S}_{1,\tau_{1},n}\left(\vartheta^{\star}\right)\left[u_{1}\right]\\
	&=-\frac{2}{3\underline{k}_{\eta_{1},\tau_{1},n}^{1-q_{1}}}\sum_{j=1}^{\underline{k}_{\eta_{1},\tau_{1},n}-2}\left(\partial_{\alpha}A\left(\bar{Y}_{\tau_{1},j-1},\alpha^{\star}\right)\right)\\
	&\hspace{4cm}\left[u_{1},\Delta_{\tau_{1},n}^{-1}\left(\bar{Y}_{\tau_{1},j+1}-\bar{Y}_{\tau_{1},j}\right)^{\otimes2}
	-\frac{2}{3}A_{\tau_{1},n}\left(\bar{Y}_{\tau_{1},j-1},\alpha^{\star},\Lambda_{\star}\right)\right]\\
	&=-\frac{2}{3\underline{k}_{\eta_{1},\tau_{1},n}^{1-q_{1}}}\sum_{j=1}^{\underline{k}_{\eta_{1},\tau_{1},n}-2}\left(\partial_{\alpha}A\left(\bar{Y}_{\tau_{1},j-1},\alpha^{\star}\right)\right)\left[u_{1},\frac{2}{3}\widehat{A}_{\tau_{1},j,n}-\frac{2}{3}A_{\tau_{1},n}\left(\bar{Y}_{\tau_{1},j-1},\alpha^{\star},\Lambda_{\star}\right)\right]\\
	&\qquad-\frac{2}{3\underline{k}_{\eta_{1},\tau_{1},n}^{1-q_{1}}}\sum_{j=1}^{\underline{k}_{\eta_{1},\tau_{1},n}-2}\left(\partial_{\alpha}A\left(\bar{Y}_{j-1},\alpha^{\star}\right)\right)\left[u_{1},\Delta_{\tau_{1},n}^{-1}\left(\bar{Y}_{\tau_{1},j+1}-\bar{Y}_{\tau_{1},j}\right)^{\otimes2}
	-\frac{2}{3}\widehat{A}_{\tau_{1},j,n}\right]\\
	&=M_{1,\tau_{1},n}+\dot{R}_{1,\tau_{1},n}+\ddot{R}_{1,\tau_{1},n},
\end{align*}
where
\begin{align*}
	M_{1,\tau_{1},n}&:=-\frac{4}{9\underline{k}_{\eta_{1},\tau_{1},n}^{1-q_{1}}}\sum_{j=1}^{\underline{k}_{\eta_{1},\tau_{1},n}-2}\left(\partial_{\alpha}A\left(\bar{Y}_{\tau_{1},j-1},\alpha^{\star}\right)\right)\left[u_{1},\widehat{A}_{\tau_{1},j,n}-A_{\tau_{1},n}\left(X_{j\Delta_{\tau_{1},n}}\right)\right],\\
	\dot{R}_{1,\tau_{1},n}&:=-\frac{2}{3\underline{k}_{\eta_{1},\tau_{1},n}^{1-q_{1}}}\sum_{j=1}^{\underline{k}_{\eta_{1},\tau_{1},n}-2}\left(\partial_{\alpha}A\left(\bar{Y}_{\tau_{1},j-1},\alpha^{\star}\right)\right)\left[u_{1},\Delta_{n}^{-1}\left(\bar{Y}_{\tau_{1},j+1}-\bar{Y}_{\tau_{1},j}\right)^{\otimes2}
	-\frac{2}{3}\widehat{A}_{\tau_{1},j,n}\right],\\
	\ddot{R}_{1,\tau_{1},n}&:=-\frac{4}{9\underline{k}_{\eta_{1},\tau_{1},n}^{1-q_{1}}}\sum_{j=1}^{\underline{k}_{\eta_{1},\tau_{1},n}-2}\left(\partial_{\alpha}A\left(\bar{Y}_{\tau_{1},j-1},\alpha^{\star}\right)\right)\left[u_{1},A_{\tau_{1},n}\left(X_{j\Delta_{\tau_{1},n}}\right)-A_{\tau,n}\left(\bar{Y}_{\tau_{1},j-1}\right)\right],
\end{align*}
$a(x):=a(x,\alpha^{\star})$ and $A_{\tau_{1},n}(x):=A_{\tau_{1},n}(x,\alpha^{\star},\Lambda_{\star})$.
$\sup_{n\in\mathbf{N}}\left\|\dot{R}_{1,\tau_{1},n}\right\|_{p}<\infty$ can be derived directly from Lemma \ref{approxQuadratic} and the assumption $\underline{k}_{\eta_{1},\tau_{1},n}^{q}\Delta_{\tau,n}\to0$.
We set $\ddot{R}_{1,i,\tau_{1},n}$ as
\begin{align*}
	\ddot{R}_{1,i,\tau_{1},n}&:=-\frac{4}{9\underline{k}_{\eta_{1},\tau_{1},n}^{1-q_{1}}}\sum_{1\le 3j+i\le \underline{k}_{\eta_{1},\tau_{1},n}-2}\left(\partial_{\alpha}A\left(\bar{Y}_{\tau_{1},3j+i-1},\alpha^{\star}\right)\right)\\
	&\hspace{5cm}\left[u_{1},A_{\tau_{1},n}\left(X_{\left(3j+i\right)\Delta_{\tau_{1},n}}\right)-A_{\tau_{1},n}\left(\bar{Y}_{3j+i-1}{\tau_{1}}\right)\right],
\end{align*}
and examine only the case $i=0$ without loss of generality. We also define $\dddot{R}_{1,0,\tau_{1},n}$ as
\begin{align*}
	\dddot{R}_{1,0,\tau_{1},n}&:=-\frac{4}{9\underline{k}_{\eta_{1},\tau_{1},n}^{1-q_{1}}}\sum_{1\le 3j\le \underline{k}_{\eta_{1},\tau_{1},n}-2}\mathbf{E}\left[\left(\partial_{\alpha}A\left(\bar{Y}_{\tau_{1},3j-1},\alpha^{\star}\right)\right)\right.\\
	&\hspace{5cm}\left.\left[u_{1},A_{\tau_{1},n}\left(X_{3j\Delta_{\tau_{1},n}}\right)-A_{\tau_{1},n}\left(\bar{Y}_{\tau_{1},3j-1}\right)\right]|\mathcal{H}_{3j-1,n}^{\tau_{1}}\right].
\end{align*}
Because of Burkholder's inequality, it holds
\begin{align*}
	\mathbf{E}\left[\left|\ddot{R}_{1,0,\tau_{1},n}-\dddot{R}_{1,0,\tau_{1},n}\right|^{p}\right]^{1/p}&\le \frac{C\left(p\right)}{\underline{k}_{\eta_{1},\tau_{1},n}^{1-2q}}<\infty.
\end{align*}
Hence it is sufficient to see $\left\|\dddot{R}_{1,0,\tau_{1},n}\right\|_{p}<\infty$,
and because of Lemma \ref{CED}, we can have
\begin{align*}
	\left\|\ddot{R}_{1,0,\tau_{1},n}\right\|_{p}&\le C\left(p\right)\underline{k}_{\eta_{1},\tau_{1},n}^{q_{1}}\Delta_{\tau_{1},n},
\end{align*}
and then $\sup_{n\in\mathbf{N}}\left\|\ddot{R}_{1,\tau_{1},n}\right\|_{p}<\infty$.
We define $M_{1,i,\tau_{1},n}$ for $i=0,1,2$ as
\begin{align*}
	&M_{1,i,\tau_{1},n}\\
	&=-\frac{4}{9\underline{k}_{\eta_{1},\tau_{1},n}^{1-q_{1}}}
	\sum_{1\le 3j+i\le \underline{k}_{\eta_{1},\tau_{1},n}-2}\left(\partial_{\alpha}A\left(\bar{Y}_{3j+i-1}{\tau_{1}},\alpha^{\star}\right)\right)\left[u_{1},\widehat{A}_{\tau_{1},3j+i,n}-A_{\tau_{1},n}\left(X_{\left(3j+i\right)\Delta_{\tau_{1},n}}\right)\right],
\end{align*}
and since the property of conditional expectation $\mathbf{E}\left[\widehat{A}_{\tau_{1},j,n}|\mathcal{H}_{j,n}^{\tau_{1}}\right]=A_{\tau_{1},n}\left(X_{j\Delta_{\tau_{1},n}}\right)$holds, Burkholder's inequality verifies
\begin{align*}
	\mathbf{E}\left[\left|M_{1,i,\tau_{1},n}\right|^{p}\right]&\le \frac{C\left(p\right)}{\underline{k}_{\eta_{1},\tau_{1},n}^{2-2q_{1}}}<\infty,
\end{align*}
for all $i$ because of the integrability.

In the second place, we show (2) holds. Let us define
\begin{align*}
	\mathbb{V}_{1,\tau_{1},n}^{\left(\dagger\right)}\left(\alpha\right)&=-\frac{1}{2\underline{k}_{\eta_{1},\tau_{1},n}}\sum_{j=1}^{\underline{k}_{\eta_{1},\tau_{1},n}-2}
	\left(\left\|\frac{2}{3}
	A_{\tau_{1},n}\left(\bar{Y}_{\tau_{1},j-1},\alpha,\Lambda_{\star}\right)\right\|^{2}
	-\left\|\frac{2}{3}
	A_{\tau_{1},n}\left(\bar{Y}_{\tau_{1},j-1},\alpha^{\star},\Lambda_{\star}\right)\right\|^{2}\right.\\
	&\hspace{3cm}\left.-2\left(\Delta_{\tau_{1},n}^{-1}\left(\bar{Y}_{\tau_{1},j+1}-\bar{Y}_{\tau_{1},j}\right)^{\otimes2}\right)\left[\frac{2}{3}A_{\tau_{1},n}\left(\bar{Y}_{\tau_{1},j-1},\alpha,\Lambda_{\star}\right)\right]\right.\\
	&\hspace{3cm}\left.+2\left(\Delta_{\tau_{1},n}^{-1}\left(\bar{Y}_{\tau_{1},j+1}-\bar{Y}_{\tau_{1},j}\right)^{\otimes2}\right)\left[\frac{2}{3}A_{\tau_{1},n}\left(\bar{Y}_{\tau_{1},j-1},\alpha^{\star},\Lambda_{\star}\right)\right]\right),
\end{align*}
and then the evaluation $\sup_{n\in\mathbf{N}}\mathbf{E}\left[\left(\sup_{\alpha\in\Theta_{1}}\underline{k}_{\eta_{1},\tau_{1},n}^{\epsilon_{1}}\left|\mathbb{V}_{1,\tau_{1},n}\left(\alpha;\vartheta^{\star}\right)-\mathbb{V}_{1,\tau_{1},n}^{\left(\dagger\right)}\left(\alpha;\vartheta^{\star}\right)\right|\right)^{p}\right]<\infty$ can be easily obtained due to $\mathbf{E}\left[\left\|\sqrt{n}\left(\hat{\Lambda}_{n}-\Lambda_{\star}\right)\right\|^{p}\right]<\infty$. We also define
\begin{align*}
	\mathbb{V}_{1,\tau_{1},n}^{\left(\ddagger\right)}\left(\alpha\right)&:=-\frac{2}{9\underline{k}_{\eta_{1},\tau_{1},n}}\sum_{j=1}^{\underline{k}_{\eta_{1},\tau_{1},n}-2}
	\left(\left\|
	A_{\tau_{1},n}\left(\bar{Y}_{\tau_{1},j-1},\alpha,\Lambda_{\star}\right)\right\|^{2}
	-\left\|
	A_{\tau_{1},n}\left(\bar{Y}_{\tau_{1},j-1},\alpha^{\star},\Lambda_{\star}\right)\right\|^{2}\right.\\
	&\hspace{4cm}\left.-3\left(A_{\tau_{1},n}\left(X_{j\Delta_{\tau_{1},n}},\alpha^{\star},\Lambda_{\star}\right)\right)\left[A_{\tau_{1},n}\left(\bar{Y}_{\tau_{1},j-1},\alpha,\Lambda_{\star}\right)\right]\right.\\
	&\hspace{4cm}\left.+3\left(A_{\tau_{1},n}\left(X_{j\Delta_{\tau_{1},n}},\alpha^{\star},\Lambda_{\star}\right)\right)\left[A_{\tau_{1},n}\left(\bar{Y}_{\tau_{1},j-1},\alpha^{\star},\Lambda_{\star}\right)\right]\right);
\end{align*}
then
\begin{align*}
	&\underline{k}_{\eta_{1},\tau_{1},n}^{\epsilon_{1}}\left(\mathbb{V}_{1,\tau_{1},n}^{\left(\dagger\right)}\left(\alpha\right)-\mathbb{V}_{1,\tau_{1},n}^{\left(\ddagger\right)}\left(\alpha\right)\right)\\
	&=\frac{2\underline{k}_{\eta_{1},\tau_{1},n}^{\epsilon_{1}}}{3\underline{k}_{\eta_{1},\tau_{1},n}}\sum_{j=1}^{\underline{k}_{\eta_{1},\tau_{1},n}-2}
	\left(\left(\Delta_{\tau_{1},n}^{-1}\left(\bar{Y}_{\tau_{1},j+1}-\bar{Y}_{\tau_{1},j}\right)^{\otimes2}-A_{\tau_{1},n}\left(X_{j\Delta_{\tau_{1},n}},\alpha^{\star},\Lambda_{\star}\right)\right)\right.\\
	&\hspace{5cm}\left[A_{\tau_{1},n}\left(\bar{Y}_{\tau_{1},j-1},\alpha,\Lambda_{\star}\right)\right]\\
	&\hspace{3cm}-\left(\Delta_{\tau_{1},n}^{-1}\left(\bar{Y}_{\tau_{1},j+1}-\bar{Y}_{\tau_{1},j}\right)^{\otimes2}-A_{\tau_{1},n}\left(X_{j\Delta_{\tau_{1},n}},\alpha^{\star},\Lambda_{\star}\right)\right)\\
	&\hspace{5cm}\left.\left[A_{\tau_{1},n}\left(\bar{Y}_{\tau_{1},j-1},\alpha^{\star},\Lambda_{\star}\right)\right]\right).
\end{align*}
Using Lemma \ref{approxQuadratic} as $\dot{R}_{1,\tau_{1},n}$, it is easy to have
\begin{align*}
	\left\|\underline{k}_{\eta_{1},\tau_{1},n}^{\epsilon_{1}}\left(\mathbb{V}_{1,\tau_{1},n}^{\left(\dagger\right)}\left(\alpha\right)-\mathbb{V}_{1,\tau_{1},n}^{\left(\ddagger\right)}\left(\alpha\right)\right)\right\|_{p}&\le C\left(p\right)\underline{k}_{\eta_{1},\tau_{1},n}^{\epsilon_{1}}\Delta_{\tau_{1},n},
\end{align*}
and similarly
\begin{align*}
	\left\|\underline{k}_{\eta_{1},\tau_{1},n}^{\epsilon_{1}}\partial_{\alpha}\left(\mathbb{V}_{1,\tau_{1},n}^{\left(\dagger\right)}\left(\alpha\right)-\mathbb{V}_{1,\tau_{1},n}^{\left(\ddagger\right)}\left(\alpha\right)\right)\right\|_{p}&\le C\left(p\right)\underline{k}_{\eta_{1},\tau_{1},n}^{\epsilon_{1}}\Delta_{\tau_{1},n};
\end{align*}
then Sobolev's inequality verifies 
\begin{align*}
	\sup_{n\in\mathbf{N}}\mathbf{E}\left[\left(\sup_{\alpha\in\Theta_{1}}\underline{k}_{\eta_{1},\tau_{1},n}^{\epsilon_{1}}\left|\mathbb{V}_{1,\tau_{1},n}^{\left(\dagger\right)}\left(\alpha;\vartheta^{\star}\right)-\mathbb{V}_{1,\tau_{1},n}^{\left(\ddagger\right)}\left(\alpha;\vartheta^{\star}\right)\right|\right)^{p}\right]<\infty.
\end{align*}
Hence it is sufficient to obtain the evaluation
\begin{align*}
	\sup_{n\in\mathbf{N}}\mathbf{E}\left[\left(\sup_{\alpha\in\Theta_{1}}\underline{k}_{\eta_{1},\tau_{1},n}^{\epsilon_{1}}\left|\mathbb{V}_{1,\tau_{1},n}\left(\alpha;\vartheta^{\star}\right)-\mathbb{V}_{1,\tau_{1},n}^{\left(\ddagger\right)}\left(\alpha;\vartheta^{\star}\right)\right|\right)^{p}\right]<\infty.
\end{align*}
Let us define $M_{1,\tau_{1},n}^{\left(\dagger\right)}$ and $R_{1,\tau_{1},n}^{\left(\dagger\right)}$ as
\begin{align*}
	M_{1,\tau_{1},n}^{\left(\dagger\right)}&=-\frac{2}{9\underline{k}_{\eta_{1},\tau_{1},n}}\sum_{j=1}^{\underline{k}_{\eta_{1},\tau_{1},n}-2}
	\left(\left\|
	A_{\tau_{1},n}\left(X_{j\Delta_{\tau_{1},n}},\alpha,\Lambda_{\star}\right)\right\|^{2}\right.\\
	&\hspace{3.5cm}-3\left(A_{\tau_{1},n}\left(X_{j\Delta_{\tau_{1},n}},\alpha^{\star},\Lambda_{\star}\right)\right)\left[A_{\tau_{1},n}\left(X_{j\Delta_{\tau_{1},n}},\alpha,\Lambda_{\star}\right)\right]\\
	&\hspace{3.5cm}-\left\|
	A_{\tau_{1},n}\left(X_{j\Delta_{\tau_{1},n}},\alpha^{\star},\Lambda_{\star}\right)\right\|^{2}\\
	&\hspace{3.5cm}\left.+3\left(A_{\tau_{1},n}\left(X_{j\Delta_{\tau_{1},n}},\alpha^{\star},\Lambda_{\star}\right)\right)\left[A_{\tau_{1},n}\left(X_{j\Delta_{\tau_{1},n}},\alpha^{\star},\Lambda_{\star}\right)\right]\right),\\
	R_{1,\tau_{1},n}^{\left(\dagger\right)}&=-\frac{2}{9\underline{k}_{\eta_{1},\tau_{1},n}}\sum_{j=1}^{\underline{k}_{\eta_{1},\tau_{1},n}-2}
	\left(\left\|
	A_{\tau_{1},n}\left(X_{j\Delta_{\tau_{1},n}},\alpha,\Lambda_{\star}\right)\right\|^{2}\right.\\
	&\hspace{3.5cm}-3\left(A_{\tau,n}\left(X_{j\Delta_{\tau_{1},n}},\alpha^{\star},\Lambda_{\star}\right)\right)\left[A_{\tau,n}\left(X_{j\Delta_{\tau_{1},n}},\alpha,\Lambda_{\star}\right)\right]\\
	&\hspace{3.5cm}-\left\|
	A_{\tau,n}\left(X_{j\Delta_{\tau_{1},n}},\alpha^{\star}\Lambda_{\star}\right)\right\|^{2}\\
	&\hspace{3.5cm}\left.+3\left(A_{\tau,n}\left(X_{j\Delta_{\tau_{1},n}},\alpha^{\star},\Lambda_{\star}\right)\right)\left[A_{\tau,n}\left(X_{j\Delta_{\tau_{1},n}},\alpha^{\star},\Lambda_{\star}\right)\right]\right)\\
	&\quad-\frac{2}{9\underline{k}_{\eta_{1},\tau_{1},n}}\sum_{j=1}^{\underline{k}_{\eta_{1},\tau_{1},n}-2}
	\left(\left\|
	A_{\tau_{1},n}\left(\bar{Y}_{\tau_{1},j-1},\alpha,\Lambda_{\star}\right)\right\|^{2}\right.\\
	&\hspace{4cm}-3\left(A_{\tau,n}\left(X_{j\Delta_{\tau_{1},n}},\alpha^{\star},\Lambda_{\star}\right)\right)\left[A_{\tau,n}\left(\bar{Y}_{\tau_{1},j-1},\alpha,\Lambda_{\star}\right)\right]\\
	&\hspace{4cm}-\left\|
	A_{\tau_{1},n}\left(\bar{Y}_{\tau_{1},j-1},\alpha^{\star}\Lambda_{\star}\right)\right\|^{2}\\
	&\hspace{4cm}\left.+3\left(A_{\tau,n}\left(X_{j\Delta_{\tau_{1},n}},\alpha^{\star},\Lambda_{\star}\right)\right)\left[A_{\tau,n}\left(\bar{Y}_{\tau_{1},j-1},\alpha^{\star},\Lambda_{\star}\right)\right]\right).
\end{align*}
$\sup_{n\in\mathbf{N}}\left\|\sup_{\alpha\in\Theta_{1}}\underline{k}_{\eta_{1},\tau_{1},n}^{\epsilon_{1}}R_{1,\tau_{1},n}^{\left(\dagger\right)}\left(\alpha\right)\right\|_{p}\le C\left(p\right)\Delta_{\tau_{1},n}^{\frac{1}{2}}$ and $\sup_{n\in\mathbf{N}}\left\|\sup_{\alpha\in\Theta_{1}}\underline{k}_{\eta_{1},\tau_{1},n}^{\epsilon_{1}}M_{1,\tau_{1},n}^{\left(\dagger\right)}\left(\alpha\right)\right\|_{p}<\infty$ are easily obtained by the discussion parallel to \citet{Nakakita-Uchida-2018c} and \citet{Yoshida-2011} respectively.

(3) and (4) are shown in the way parallel to \citet{Nakakita-Uchida-2018c} and the discussion above respectively.
\end{proof}

\begin{proof}[Proof of Lemma \ref{initBetaLemma}]
We decompose
\begin{align*}
	S_{2,\tau_{2},n}\left(\vartheta^{\star}\right)&:=-\frac{1}{\underline{T}_{\eta_{2},n}^{1-q_{2}}}\sum_{j=1}^{\underline{k}_{\eta_{2},\tau_{2},n}-2}\left(\partial_{\beta}b\left(\bar{Y}_{\tau_{2},j-1},\beta^{\star}\right)\right)
	\left[\bar{Y}_{\tau_{2},j+1}-\bar{Y}_{\tau_{2},j}-\Delta_{\tau_{2},n}b\left(\bar{Y}_{\tau_{2},j-1},\beta^{\star}\right)\right]\\
	&=-\frac{1}{\underline{T}_{\eta_{2},n}^{1-q_{2}}}\sum_{j=1}^{\underline{k}_{\eta_{2},\tau_{2},n}-2}\left(\partial_{\beta}b\left(\bar{Y}_{\tau_{2},j-1},\beta^{\star}\right)\right)
	\left[a\left(X_{j\Delta_{\tau_{2},n}}\right)\left(\zeta_{\tau_{2},j+1,n}+\zeta_{\tau_{2},j+2,n}'\right)\right]\\
	&\qquad-\frac{1}{\underline{T}_{\eta_{2},n}^{1-q_{2}}}\sum_{j=1}^{\underline{k}_{\eta_{2},\tau_{2},n}-2}\left(\partial_{\beta}b\left(\bar{Y}_{\tau_{2},j-1},\beta^{\star}\right)\right)
	\left[\left(\Lambda_{\star}\right)^{1/2}\left(\bar{\varepsilon}_{\tau_{2},j+1}-\bar{\varepsilon}_{\tau_{2},j}\right)\right]\\
	&\qquad-\frac{1}{\underline{T}_{\eta_{2},n}^{1-q_{2}}}\sum_{j=1}^{\underline{k}_{\eta_{2},\tau_{2},n}-2}\left(\partial_{\beta}b\left(\bar{Y}_{\tau_{2},j-1},\beta^{\star}\right)\right)
	\left[e_{\tau_{2},j,n}\right]\\
	&=M_{2,\tau_{2},n}+R_{2,\tau_{2},n},
\end{align*}
where
\begin{align*}
	M_{2,\tau_{2},n}&:=-\frac{1}{\underline{T}_{\eta_{2},n}^{1-q_{2}}}\sum_{j=1}^{\underline{k}_{\eta_{2},\tau_{2},n}-2}\left(\partial_{\beta}b\left(\bar{Y}_{\tau_{2},j-1},\beta^{\star}\right)\right)
	\left[a\left(X_{j\Delta_{\tau_{2},n}}\right)\left(\zeta_{\tau_{2},j+1,n}+\zeta_{\tau_{2},j+2,n}'\right)\right]\\
	&\qquad-\frac{1}{\underline{T}_{\eta_{2},n}^{1-q_{2}}}\sum_{j=1}^{\underline{k}_{\eta_{2},\tau_{2},n}-2}\left(\partial_{\beta}b\left(\bar{Y}_{\tau_{2},j-1},\beta^{\star}\right)\right)
	\left[\left(\Lambda_{\star}\right)^{1/2}\left(\bar{\varepsilon}_{\tau_{2},j+1}-\bar{\varepsilon}_{\tau_{2},j}\right)\right],\\
	R_{2,\tau_{2},n}&=-\frac{1}{\underline{T}_{\eta_{2},n}^{1-q_{2}}}\sum_{j=1}^{\underline{k}_{\eta_{2},\tau_{2},n}-2}\left(\partial_{\beta}b\left(\bar{Y}_{\tau_{2},j-1},\beta^{\star}\right)\right)
	\left[e_{\tau_{2},j,n}\right].
\end{align*}
Since Burkholder's inequality verifies the evaluations such that
\begin{align*}
	&\mathbf{E}\left[\left|\frac{1}{\underline{T}_{\eta_{2},n}^{1-q_{2}}}\sum_{1\le 3j\le \underline{k}_{\eta_{2},\tau_{2},n}-2}\left(\partial_{\beta}b\left(\bar{Y}_{\tau_{2},3j-1},\beta^{\star}\right)\right)
	\left[a\left(X_{3j\Delta_{\tau_{2},n}}\right)\left(\zeta_{\tau_{2},3j+1,n}+\zeta_{\tau_{2},3j+2,n}'\right)\right]\right|^{p}\right]\\
	&\le \frac{C\left(p\right)}{\underline{T}_{\eta_{2},n}^{p-pq_{2}}}\mathbf{E}\left[\left|\sum_{1\le 3j\le \underline{k}_{\eta_{2},\tau_{2},n}-2}\left|\partial_{\beta}b\left(\bar{Y}_{\tau_{2},3j-1},\beta^{\star}\right)\right|^{2}
	\left|a\left(X_{3j\Delta_{\tau_{2},n}}\right)\left(\zeta_{\tau_{2},3j+1,n}+\zeta_{\tau_{2},3j+2,n}'\right)\right|^{2}\right|^{p/2}\right]\\
	&\le \frac{C\left(p\right)\underline{T}_{\eta_{2},n}^{\frac{p}{2}}}{\underline{T}_{\eta_{2},n}^{p\left(1-q_{2}\right)}},
\end{align*}
and
\begin{align*}
	\mathbf{E}\left[\left|\frac{1}{\underline{T}_{\eta_{2},n}^{1-q_{2}}}\sum_{1\le 3j\le \underline{k}_{\eta_{2},\tau_{2},n}-2}\left(\partial_{\beta}b\left(\bar{Y}_{\tau_{2},3j-1},\beta^{\star}\right)\right)
	\left[\left(\Lambda_{\star}\right)^{1/2}\left(\bar{\varepsilon}_{\tau_{2},j+1}-\bar{\varepsilon}_{\tau_{2},j}\right)\right]\right|^{p}\right]\le \frac{C\left(p\right)T_{n}^{\frac{p}{2}}}{\underline{T}_{\eta_{2},n}^{p\left(1-q_{2}\right)}},
\end{align*}
we obtain $\sup_{n\in\mathbf{N}}\left\|M_{2,\tau_{2},n}\right\|_{p}<\infty$. With respect to $R_{2,\tau_{2},n}$, Burkholder's inequality also leads to
\begin{align*}
	&\mathbf{E}\left[\left|\frac{1}{\underline{T}_{\eta_{2},n}^{1-q_{2}}}\sum_{1\le 3j \le \underline{k}_{\eta_{2},\tau_{2},n}-2}\left(\partial_{\beta}b\left(\bar{Y}_{\tau_{2},3j-1},\beta^{\star}\right)\right)
	\left[e_{\tau_{2},3j,n}-\mathbf{E}\left[e_{\tau_{2},3j,n}|\mathcal{H}_{3j,n}^{\tau_{2}}\right]\right]\right|^{p}\right]\\
	&\le C\left(p\right)\mathbf{E}\left[\left|\frac{1}{\underline{T}_{\eta_{2},n}^{2-2q_{2}}}\sum_{1\le 3j \le \underline{k}_{\eta_{2},\tau_{2},n}-2}\left|\partial_{\beta}b\left(\bar{Y}_{\tau_{2},3j-1},\beta^{\star}\right)\right|^{2}
	\left|e_{\tau_{2},3j,n}-\mathbf{E}\left[e_{\tau_{2},3j,n}|\mathcal{H}_{3j,n}^{\tau_{2}}\right]\right|^{2}\right|^{p/2}\right]\\
	&\le\frac{C\left(p\right)\underline{k}_{\eta_{2},\tau_{2},n}^{\frac{p}{2}}\Delta_{\tau_{2},n}^{p}}{\underline{T}_{\eta_{2},n}^{p\left(1-q_{2}\right)}}
%	\\
%	&
	\le C\left(p\right)\left[\underline{k}_{\eta_{2},\tau_{2},n}^{q_{2}-\frac{1}{2}}\Delta_{\tau_{2},n}^{q_{2}}\right]^{p},
\end{align*}
and we also have
\begin{align*}
&	\mathbf{E}\left[\left|\frac{1}{\underline{T}_{\eta_{2},n}^{1-q_{2}}}\sum_{1\le 3j \le \underline{k}_{\eta_{2},\tau_{2},n}-2}\left(\partial_{\beta}b\left(\bar{Y}_{\tau_{2},3j-1},\beta^{\star}\right)\right)
	\left[\mathbf{E}\left[e_{\tau_{2},3j,n}|\mathcal{H}_{3j,n}^{\tau_{2}}\right]\right]\right|^{p}\right]
\\
	&\le \frac{C\left(p\right)\underline{k}_{\eta_{2},\tau_{2},n}^{p}\Delta_{\tau_{2},n}^{2p}}{\underline{T}_{\eta_{2},n}^{p\left(1-q_{2}\right)}}
%	\\
%	&
	\le \frac{C\left(p\right)\underline{k}_{\eta_{2},\tau_{2},n}^{p}\Delta_{\tau_{2},n}^{2p}}{\left(\underline{k}_{\eta_{2},\tau_{2},n}\Delta_{\tau_{2},n}\right)^{p
			\left(1-q_{2}\right)}}
%	\\
%	&
	\le C\left(p\right)\left[\underline{k}_{\eta_{2},\tau_{2},n}^{q_{2}}\Delta_{\tau_{2},n}^{1+q_{2}}\right]^{p},
\end{align*}
which is led by  $\left|\mathbf{E}\left[e_{\tau_{2},j,n}|\mathcal{H}_{j,n}^{\tau_{2}}\right]\right|\le C\Delta_{\tau_{2},n}^{2}\left(1+\left|X_{j\Delta_{\tau_{2},n}}\right|\right)^{C}$ shown in \citet{Nakakita-Uchida-2018c}. The derivation of (ii) is parallel to that of (ii) in Lemma 1.

(3) and (4) are shown 
%in  a parallel
in an analogous  way 
to \citet{Nakakita-Uchida-2018c} and the discussion above respectively.
\end{proof}

\begin{proof}[Proof of Theorem \ref{PLDI}]
Lemma \ref{initAlphaLemma}, Lemma \ref{initBetaLemma} and Theorem 3 in \citet{Yoshida-2011} lead to the PLDI as discussed in \citet{Nakakita-Uchida-2018c}. The $L^{M}$-evaluation of estimators also can be obtained with a parallel discussion to \citet{Nakakita-Uchida-2018c}.
\end{proof}

\begin{proof}[Proof of Theorem \ref{result}]
First of all, we show the $L^{p}$-boundedness
\begin{align*}
    \mathbf{E}\left[\left|\sqrt{k_{\tau_{3},n}}\left(\hat{\alpha}_{J_{1},n}-\alpha^{\star}\right)\right|^{p}\right]<\infty,
    \mathbf{E}\left[\left|\sqrt{T_{n}}\left(\hat{\beta}_{J_{2},n}-\beta^{\star}\right)\right|^{p}\right]<\infty.
\end{align*}
As discussed in \citet{Kaino-Uchida-2018b}, for all $k=1,\ldots,J_{1}$, on $K_{n}\left(\hat{\alpha}_{k-1,n}\right)$,
\begin{align*}
    &\partial_{\alpha}\mathbb{H}_{1,\tau_{3},n}\left(\alpha^{\star}|\hat{\Lambda}_{n}\right)\\
    &= \partial_{\alpha}\mathbb{H}_{1,\tau_{3},n}\left(\hat{\alpha}_{k-1,n}|\hat{\Lambda}_{n}\right) + \partial_{\alpha}^{2}\mathbb{H}_{1,\tau_{3},n}\left(\hat{\alpha}_{k-1,n}|\hat{\Lambda}_{n}\right)\left[\alpha^{\star}-\hat{\alpha}_{k-1,n}\right]\\
    &\quad + \int_{0}^{1}\left(1-s\right)\partial_{\alpha}^{3}\mathbb{H}_{1,\tau_{3},n}\left(\hat{\alpha}_{k-1,n}+s\left(\alpha^{\star}-\hat{\alpha}_{k-1,n}\right)|\hat{\Lambda}_{n}\right)\mathrm{d}s\left[\left(\alpha^{\star}-\hat{\alpha}_{k-1,n}\right)^{\otimes2}\right],
\end{align*}
and
\begin{align*}
    \hat{\alpha}_{k,n}&=\hat{\alpha}_{k-1,n} - \left(\frac{1}{k_{\tau_{3},n}}\partial_{\alpha}^{2}\mathbb{H}_{1,\tau_{3},n}\left(\hat{\alpha}_{k-1,n}|\hat{\Lambda}_{n}\right)\right)^{-1}\frac{1}{k_{\tau_{3},n}}\partial_{\alpha}\mathbb{H}_{1,\tau_{3},n}\left(\hat{\alpha}_{k-1,n}|\hat{\Lambda}_{n}\right)\\
    &=\hat{\alpha}_{k-1,n} - \left(\frac{1}{k_{\tau_{3},n}}\partial_{\alpha}^{2}\mathbb{H}_{1,\tau_{3},n}\left(\hat{\alpha}_{k-1,n}|\hat{\Lambda}_{n}\right)\right)^{-1}\frac{1}{k_{\tau_{3},n}}\partial_{\alpha}\mathbb{H}_{1,\tau_{3},n}\left(\alpha^{\star}|\hat{\Lambda}_{n}\right)\\
    &\quad+\left(\alpha^{\star}-\hat{\alpha}_{k-1,n}\right)\\
    &\quad+\left(\frac{1}{k_{\tau_{3},n}}\partial_{\alpha}^{2}\mathbb{H}_{1,\tau_{3},n}\left(\hat{\alpha}_{k-1,n}|\hat{\Lambda}_{n}\right)\right)^{-1}\\
    &\qquad\times\frac{1}{k_{\tau_{3},n}}\int_{0}^{1}\left(1-s\right)\partial_{\alpha}^{3}\mathbb{H}_{1,\tau_{3},n}\left(\hat{\alpha}_{k-1,n}+s\left(\alpha^{\star}-\hat{\alpha}_{k-1,n}\right)|\hat{\Lambda}_{n}\right)\mathrm{d}s\left[\left(\alpha^{\star}-\hat{\alpha}_{k-1,n}\right)^{\otimes2}\right],
\end{align*}
that is,
\begin{align*}
    &\hat{\alpha}_{k,n}-\alpha^{\star}\\
    &= - \left(\frac{1}{k_{\tau_{3},n}}\partial_{\alpha}^{2}\mathbb{H}_{1,\tau_{3},n}\left(\hat{\alpha}_{k-1,n}|\hat{\Lambda}_{n}\right)\right)^{-1}\frac{1}{k_{\tau_{3},n}}\partial_{\alpha}\mathbb{H}_{1,\tau_{3},n}\left(\alpha^{\star}|\hat{\Lambda}_{n}\right)\\
    &\quad+\left(\frac{1}{k_{\tau_{3},n}}\partial_{\alpha}^{2}\mathbb{H}_{1,\tau_{3},n}\left(\hat{\alpha}_{k-1,n}|\hat{\Lambda}_{n}\right)\right)^{-1}\\
    &\qquad\times\frac{1}{k_{\tau_{3},n}}\int_{0}^{1}\left(1-s\right)\partial_{\alpha}^{3}\mathbb{H}_{1,\tau_{3},n}\left(\hat{\alpha}_{k-1,n}+s\left(\alpha^{\star}-\hat{\alpha}_{k-1,n}\right)|\hat{\Lambda}_{n}\right)\mathrm{d}s\left[\left(\alpha^{\star}-\hat{\alpha}_{k-1,n}\right)^{\otimes2}\right].
\end{align*}
%The discussion parallel to that in 
The analogous argument of \citet{Kamatani-Uchida-2015} verifies
\begin{align*}
    \mathbf{E}\left[\left\|\left(\frac{1}{k_{\tau_{3},n}}\partial_{\alpha}^{2}\mathbb{H}_{1,\tau_{3},n}\left(\hat{\alpha}_{k-1,n}|\hat{\Lambda}_{n}\right)\right)^{-1}\right\|^{p}\mathbf{1}_{K_{n}\left(\hat{\alpha}_{k-1,n}\right)}\right]<\infty,\ P\left(K_{n}^{c}\left(\hat{\alpha}_{k-1,n}\right)\right)\le \frac{C(L)}{k_{\tau_{3},n}^{L}}.
\end{align*}
Then we obtain the $L^{p}$-boundedness for any $p\ge 1$ and $k\in\mathbf{N}$ such that $2^{k}q_{1}'\le 1/2$, where $q_{1}'=q_{1}\left(\eta_{1}-\gamma/\tau_{1}\right)/\left(1-\gamma'/\tau_{3}\right)$, that is to say, $k\le -\log_{2}q_{1}'-1$, as the discussion in \citet{Kamatani-Uchida-2015},
\begin{align*}
    &\mathbf{E}\left[\left|k_{\tau_{3},n}^{2^{k}q_{1}'}\left(\hat{\alpha}_{k,n}-\alpha^{\star}\right)\right|^{p}\right]\\
    &\le C\left(p\right)\mathbf{E}\left[\left\|\left(\frac{1}{k_{\tau_{3},n}}\partial_{\alpha}^{2}\mathbb{H}_{1,\tau_{3},n}\left(\hat{\alpha}_{k-1,n}|\hat{\Lambda}_{n}\right)\right)^{-1}\right\|^{2p}\mathbf{1}_{K_{n}\left(\hat{\alpha}_{k-1,n}\right)}\right]^{1/2}\\
    &\qquad\times\mathbf{E}\left[\left|\frac{k_{\tau_{3},n}^{2^{k}q_{1}'}}{k_{\tau_{3},n}}\partial_{\alpha}\mathbb{H}_{1,\tau_{3},n}\left(\alpha^{\star}|\hat{\Lambda}_{n}\right)\right|^{2p}\right]^{1/2}\\
    &\quad+C\left(p\right)\mathbf{E}\left[\left\|\left(\frac{1}{k_{\tau_{3},n}}\partial_{\alpha}^{2}\mathbb{H}_{1,\tau_{3},n}\left(\hat{\alpha}_{k-1,n}|\hat{\Lambda}_{n}\right)\right)^{-1}\right\|^{2p}\mathbf{1}_{K_{n}\left(\hat{\alpha}_{k-1,n}\right)}\right]^{1/2}\\
    &\qquad\times\mathbf{E}\left[\left\|\frac{1}{k_{\tau_{3},n}}\int_{0}^{1}\left(1-s\right)\partial_{\alpha}^{3}\mathbb{H}_{1,\tau_{3},n}\left(\hat{\alpha}_{k-1,n}+s\left(\alpha^{\star}-\hat{\alpha}_{k-1,n}\right)|\hat{\Lambda}_{n}\right)\mathrm{d}s\right\|^{4p}\right]^{1/4}\\
    &\qquad\times \mathbf{E}\left[\left|k_{\tau_{3},n}^{2^{k-1}q_{1}'}\left(\alpha^{\star}-\hat{\alpha}_{k-1,n}\right)\right|^{8p}\right]^{1/4}\\
    &\quad+C\left(p\right).
\end{align*}
The result of \citet{Nakakita-Uchida-2018c} leads to the $L^{p}$-boundednesses for all $p\ge1$ such as
\begin{align*}
    &\mathbf{E}\left[\left|\frac{k_{\tau_{3},n}^{2^{k}q_{1}'}}{k_{\tau_{3},n}}\partial_{\alpha}\mathbb{H}_{1,\tau_{3},n}\left(\alpha^{\star}|\hat{\Lambda}_{n}\right)\right|^{p}\right]<\infty,\\
    &\mathbf{E}\left[\left\|\frac{1}{k_{\tau_{3},n}}\int_{0}^{1}\left(1-s\right)\partial_{\alpha}^{3}\mathbb{H}_{1,\tau_{3},n}\left(\hat{\alpha}_{k-1,n}+s\left(\alpha^{\star}-\hat{\alpha}_{k-1,n}\right)|\hat{\Lambda}_{n}\right)\mathrm{d}s\right\|^{p}\right]<\infty.
\end{align*}
Hence we obtain
\begin{align*}
    \mathbf{E}\left[\left|k_{\tau_{3},n}^{2^{k-1}q_{1}'}\left(\hat{\alpha}_{k-1,n}-\alpha^{\star}\right)\right|^{p}\right]<\infty\Rightarrow \mathbf{E}\left[\left|k_{\tau_{3},n}^{2^{k}q_{1}'}\left(\hat{\alpha}_{k,n}-\alpha^{\star}\right)\right|^{p}\right]<\infty,
\end{align*}
and as shown in Theorem \ref{PLDI}, $\mathbf{E}\left[\left|k_{\tau_{3},n}^{ q_{1}'}\left(\hat{\alpha}_{0,n}-\alpha^{\star}\right)\right|^{p}\right]<\infty$, and therefore we have the $L^{p}$-evaluation $\mathbf{E}\left[\left|k_{\tau_{3},n}^{2^{J_{1}-1}q_{1}'}\left(\hat{\alpha}_{J_{1},n}-\alpha^{\star}\right)\right|^{p}\right]<\infty$ as a result. Then
\begin{align*}
    &\mathbf{E}\left[\left|k_{\tau_{3},n}^{1/2}\left(\hat{\alpha}_{J_{1},n}-\alpha^{\star}\right)\right|^{p}\right]\\
    &\le C\left(p\right)\mathbf{E}\left[\left\|\left(\frac{1}{k_{\tau_{3},n}}\partial_{\alpha}^{2}\mathbb{H}_{1,\tau_{3},n}\left(\hat{\alpha}_{J_{1}-1,n}|\hat{\Lambda}_{n}\right)\right)^{-1}\right\|^{2p}\mathbf{1}_{K_{n}\left(\hat{\alpha}_{J_{1}-1,n}\right)}\right]^{1/2}\\
    &\hspace{2cm}\times\mathbf{E}\left[\left|\frac{k_{\tau_{3},n}^{1/2}}{k_{\tau_{3},n}}\partial_{\alpha}\mathbb{H}_{1,\tau_{3},n}\left(\alpha^{\star}|\hat{\Lambda}_{n}\right)\right|^{2p}\right]^{1/2}\\
    &\quad+C\left(p\right)\mathbf{E}\left[\left\|\left(\frac{1}{k_{\tau_{3},n}}\partial_{\alpha}^{2}\mathbb{H}_{1,\tau_{3},n}\left(\hat{\alpha}_{J_{1}-1,n}|\hat{\Lambda}_{n}\right)\right)^{-1}\right\|^{2p}\mathbf{1}_{K_{n}\left(\hat{\alpha}_{J_{1}-1,n}\right)}\right]^{1/2}\\
    &\qquad\times\mathbf{E}\left[\left\|\frac{1}{k_{\tau_{3},n}}\int_{0}^{1}\left(1-s\right)\partial_{\alpha}^{3}\mathbb{H}_{1,\tau_{3},n}\left(\hat{\alpha}_{J_{1}-1,n}+s\left(\alpha^{\star}-\hat{\alpha}_{J_{1}-1,n}\right)|\hat{\Lambda}_{n}\right)\mathrm{d}s\right\|^{4p}\right]^{1/4}\\
    &\qquad\times \mathbf{E}\left[\left|k_{\tau_{3},n}^{1/4}\left(\alpha^{\star}-\hat{\alpha}_{J_{1}-1,n}\right)\right|^{8p}\right]^{1/4}\\
    &\quad+C\left(p\right)\\
    &<\infty
\end{align*}
because the discussion above holds, $J_{1}> -\log_{2}q_{1}'-1$ and hence $2^{J_{1}-1}q_{1}'> 1/4$. The same result for $\hat{\beta}_{J_{2},n}$ can be derived from the parallel discussion.

In the second place, we will see the convergence in law. Actually it is sufficient to see
\begin{align*}
    k_{\tau_{3},n}^{1/4}\left(\hat{\alpha}_{J_{1}-1,n}-\alpha^{\star} \right) \to^{P} \mathbf{0},
\end{align*}
and it can be verified because
\begin{align*}
    \mathbf{E}\left[\left|k_{\tau_{3},n}^{1/4}\left(\alpha^{\star}-\hat{\alpha}_{J_{1}-1,n}\right)\right|^{p}\right]=k_{\tau_{3},n}^{1/4-2^{J_{1}-1}q_{1}'}\mathbf{E}\left[\left|k_{\tau_{3},n}^{2^{J_{1}-1}q_{1}'}\left(\alpha^{\star}-\hat{\alpha}_{J_{1}-1,n}\right)\right|^{p}\right]\to 0.
\end{align*}
The same argument holds for $\hat{\beta}_{J_{2},n}$ too.
\end{proof}

%\bibliography{bib180930}
%\bibliographystyle{apalike}

\end{document}